\documentclass[10pt, conference, compsocconf]{IEEEtran}
\usepackage{cite}

\ifCLASSINFOpdf
\else
\fi
\usepackage[cmex10]{amsmath}
\usepackage{algorithmic}
\usepackage{amsfonts}
\usepackage{algorithm}
\usepackage{graphicx}
\newtheorem{thm}{Theorem}
\newtheorem{lem}[thm]{Lemma}

\begin{document}
%
\title{Fast Parallel Randomized QR with Column Pivoting Algorithms for \\ Reliable Low-rank Matrix Approximations}

\author{\IEEEauthorblockN{Jianwei Xiao}
\IEEEauthorblockA{Department of Mathematics\\
UC Berkeley\\
Berkeley, CA, USA\\
jwxiao@berkeley.edu}
\and
\IEEEauthorblockN{Ming Gu}
\IEEEauthorblockA{Department of Mathematics\\
UC Berkeley\\
Berkeley, CA, USA\\
mgu@math.berkeley.edu}
\and
\IEEEauthorblockN{Julien Langou}
\IEEEauthorblockA{Department of Mathematical and Statistical Sciences\\
University of Colorado Denver\\
Denver, CO, USA\\
julien.langou@ucdenver.edu}}


%


\maketitle
{\let\thefootnote\relax\footnotetext{This paper has been accepted by 2017 IEEE 24th International Conference on High Performance Computing (HiPC) and awarded the best paper prize.}}

\begin{abstract}
Factorizing large matrices by QR with column pivoting (QRCP) is substantially more expensive than QR without pivoting, owing to communication costs required for pivoting decisions. In contrast, randomized QRCP (RQRCP) algorithms have proven themselves empirically to be highly competitive with high-performance implementations of QR in processing time, on uniprocessor and shared memory machines, and as reliable as QRCP in pivot quality. 

We show that RQRCP algorithms can be as reliable as QRCP with failure probabilities exponentially decaying in oversampling size. We also analyze efficiency differences among different RQRCP algorithms. More importantly, we develop distributed memory implementations of RQRCP that are significantly better than QRCP implementations in ScaLAPACK. 

As a further development, we introduce the concept of and develop algorithms for computing spectrum-revealing QR factorizations for low-rank matrix approximations, and demonstrate their effectiveness against leading low-rank approximation methods in both theoretical and numerical reliability and efficiency.
\end{abstract}


%
\IEEEpeerreviewmaketitle

\section{Introduction}
\subsection{QR factorizations with column pivoting}
A QR factorization of a matrix $A \in \mathbb{R}^{m \times n}$ is $A=QR$ where $Q\in \mathbb{R}^{m \times m}$ is orthogonal and $R \in \mathbb{R}^{m \times n}$ is upper trapezoidal. In LAPACK \cite{anderson1999lapack} and ScaLAPACK \cite{blackford1997scalapack}, the QR factorization can be computed by xGEQRF and PxGEQRF, respectively, where $x$ indicates the matrix data type.

In practical situations where either the matrix $A$ is not always known to be of full rank or we want to find representative columns of $A$, we compute a full or partial QR factorization with column pivoting (QRCP) of the form
\begin{equation}\label{Eqn:QRCP}
A\Pi = QR
\end{equation}
for a matrix $A \in \mathbb{R}^{m \times n}$, where $\Pi\in \mathbb{R}^{n \times n}$ is a permutation matrix and $Q \in \mathbb{R}^{m \times m}$ is an orthogonal matrix. A full QRCP, with $R\in \mathbb{R}^{m \times n}$ being upper trapezoidal, can be computed by either xGEQPF or xGEQP3 in LAPACK \cite{anderson1999lapack}, and either PxGEQPF or PxGEQP3 in ScaLAPACK \cite{blackford1997scalapack}. xGEQP3 and PxGEQP3 are Level 3 BLAS versions of xGEQPF and PxGEQPF respectively. They are considerably faster, while maintaining same numerical behavior. 

Given a target rank $1 \le k \le \min \left(m,n\right)$ in partial QRCP, equation \eqref{Eqn:QRCP} can be written in a $2 \times 2$ block form as
\begin{equation*}
A\Pi
\!=\!
Q\left(
\begin{array}{cc}
R_{11} & R_{12}\\
0 & R_{22}\\
\end{array}
\right)
\!=\!
\left(
\begin{array}{cc}
Q_1 & Q_2
\end{array}
\right)\left(
\begin{array}{cc}
R_{11} & R_{12}\\
0 & R_{22}\\
\end{array}
\right),
\end{equation*} 
\noindent with upper triangular $R_{11} \in \mathbb{R}^{k \times k}$. If $\Pi$ is chosen appropriately, the partial QRCP can separate linearly independent columns from dependent ones, yielding a low-rank approximation $A \approx Q_1\left(
\begin{array}{cc}
R_{11} & R_{12}
\end{array}
\right)\Pi^T$. Efficient and reliable low-rank approximations are useful in many applications including data mining \cite{xiao2016spectrum} and image processing \cite{su2014color}.

\subsection{Randomization in numerical linear algebra}
Traditional matrix algorithms are prohibitively expensive for many applications where the datasets are very large. Randomization allows us to design provably accurate algorithms for matrix problems that are massive or computationally expensive. Randomized matrix algorithms have been successfully developed in fast least squares \cite{clarkson2013low}, sketching algorithms \cite{demmel2015communication}, low-rank approximation problems \cite{gu2015subspace}, etc.

The computational bottleneck of QRCP is in searching pivots, which requires updating all the column norms in the trailing matrix. While the number of floating point operations (flops) is relatively small, column norm updating incurs excessive communication costs and is at least as expensive as QR. \cite{duersch2015true} develops a randomized QRCP (RQRCP), where random sampling is used for column selection, resulting in dramatically reduced communication costs. They also introduce updating formulas to reduce the cost of sampling. With their column selection mechanism they obtain approximations that are comparable to those from the QRCP in quality, but with performance near QR. They demonstrate strong parallel scalability on shared memory multiple core systems using an implementation in Fortran with OpenMP. 

\subsection{Contributions}
\begin{itemize}
\item \textbf{Reliability analysis:} We show, with a rigorous probability analysis, that RQRCP algorithms can be as reliable as QRCP up to failure probabilities that exponentially decay with oversampling size. 

\item \textbf{Distributed memory implementation:} We extend RQRCP shared memory implementation of \cite{duersch2015true} to distributed memory machines. Based on ScaLAPACK, our implementation is significantly faster than QRCP routines in ScaLAPACK, yet as effective in quality. 

\item \textbf{Spectrum-revealing QR factorization:} We propose a novel variant of the QR factorization for low-rank approximation: spectrum-revealing QR factorization (SRQR). We prove singular value bounds and residual error bounds for SRQR, and develop RQRCP-based efficient algorithms for its computation. We also propose SRQR based CUR and CX matrix decomposition algorithms. SRQR based algorithms are as effective as other state-of-the-art CUR and CX matrix decomposition algorithms, while significantly faster. 
\end{itemize}

\section{Introduction to QRCP}
\begin{algorithm}
\caption{QR with column pivoting (QRCP)}
\label{algorithm_Partial Householder QR with Column Pivoting (QRCP)}
\begin{algorithmic}
\STATE $\textbf{Inputs:}$
\STATE $A$ is $m \times n$ matrix, $k$ is target rank, $1 \le k \le \min\left(m,n\right)$
\STATE $\textbf{Outputs:}$
\STATE $Q$ is $m \times m$ orthogonal matrix
\STATE $R$ is $m \times n$ upper trapezoidal matrix
\STATE $\Pi$ is $n \times n$ permutation matrix such that $A\Pi = QR$
\STATE $\textbf{Algorithm:}$
\STATE Initialize $\Pi^{\left(0\right)}=I_{n}$, $r_s = ||A\left(1:m,s\right)||_2 \; \left(1 \le s \le n\right)$
\FOR{$i=1:k$}
\STATE Find $j = argmax_{i \le s \le n} r_s$
\STATE Swap $r_i$ and $r_j$, $A\left(1:m,i\right)$ and $A\left(1:m,j\right)$
\STATE Update permutation with last swap $\Pi^{\left(i\right)}=\Pi^{\left(i-1\right)}\Pi_{i,j}$
\STATE Form Householder reflection $H_i$ from $A\left(i:m,i\right)$
\STATE Update $A\left(i:m,i:n\right) \leftarrow H_i A\left(i:m,i:n\right)$
\STATE Update $r_s = ||A\left(i+1:m,s\right)||_2 \quad \left(i+1 \le s \le n\right)$
\ENDFOR
\STATE $Q$ = $H_1H_2 \cdots H_k$ is the product of all reflections
\STATE $R$ = upper trapezoidal part of $A$, $\Pi$ = $\Pi^{\left(k\right)}$
\end{algorithmic}
\end{algorithm}
QRCP is introduced in Algorithm \ref{algorithm_Partial Householder QR with Column Pivoting (QRCP)}. In each loop, QRCP swaps the leading column with the column in the trailing matrix with the largest column norm. It is a very effective practical tool for low-rank matrix approximations, but may require additional column interchanges to reveal the rank for contrived pathological matrices \cite{gu1996efficient,kahan1966numerical}. 

To find a correct pivot in the $\left(i+1\right)^{\mbox{th}}$ loop, trailing column norms must be updated to remove contributions from row $i$. This computation, while relatively minor flop-wise, requires accessing the entire trailing matrix, and is primary cause of significant slow-down of QRCP over QR. 

LAPACK subroutines xGEQPF and xGEQP3 are based on Algorithm \ref{algorithm_Partial Householder QR with Column Pivoting (QRCP)}. The difference is that xGEQP3 re-organizes the computations to apply Householder reflections in blocks to the trailing matrix, partly with Level 3 BLAS instead of Level 2 BLAS, for faster execution. ScaLAPACK subroutines PxGEQPF and PxGEQP3 are the parallel versions of xGEQPF and xGEQP3 in distributed memory systems, respectively, with PxGEQP3 being the more efficient. 

\section{Randomized QRCP}
\subsection{Introduction to RQRCP}
\begin{algorithm}
\caption{Randomized QRCP (RQRCP)}\label{Algorithm_Partial QRCP with random sampling (RQRCP)}
\begin{algorithmic}
\STATE $\textbf{Inputs:}$
\STATE $A$ is $m \times n$ matrix, $k$ is target rank, $1 \le k \le \min\left(m,n\right)$
\STATE $\textbf{Outputs:}$
\STATE $Q$ is $m \times m$ orthogonal matrix
\STATE $R$ is $m \times n$ upper trapezoidal matrix
\STATE $\Pi$ is $n \times n$ permutation matrix such that $A\Pi = QR$
\STATE $\textbf{Algorithm:}$
\STATE Determine block size $b$ and oversampling size $p \ge 0$
\STATE Generate i.i.d Gaussian matrix $\Omega \in \mathcal{N}(0,1)^{(b+p) \times m}$
\STATE Compute $B=\Omega A$, initialize $\Pi=I_{n}$
\FOR{$i=1:b:k$}
\STATE $b=\min\left(b,k-i+1\right)$
\STATE Partial QRCP on $\!B(:,i:n)$ to get $b$ pivots $\!\left(j_1,\dots,j_{b}\right)$
\STATE Swap $\left(j_1,\dots,j_{b}\right)$ and $\left(i:i+b-1\right)$ columns in $A$, $\Pi$
\STATE Do unpivoted QR on $A\left(i:m,i:i+b-1\right) = \widetilde{Q} \widetilde{R}$
\STATE $A(i:m,i+b:n) \leftarrow \widetilde{Q}^T A(i:m,i+b:n)$
\STATE $B(1:b,i+b:n)=B(1:b,i+b:n)-B(1:b,i:j+b-1)(A(i:i+b-1.i:i+b-1))^{-1}A(i:i+b-1,i+b:n)$
\ENDFOR
\STATE Q is the product of $\widetilde{Q}$, R = upper trapezoidal part of $A$
\end{algorithmic}
\end{algorithm}
RQRCP applies a random matrix $\Omega$ of independent and identically distributed (i.i.d.) random variables to $A$ to compress $A$ into $B  = \Omega A$ with much smaller row dimension, where the block size $b$ and oversampling size $p$ are chosen with $b+p \ll m$. QRCP and RQRCP make pivot decisions on $A$ and $B$, respectively. Since $B$ has much smaller row dimension than $A$, RQRCP can choose pivots much more quickly than QRCP. RQRCP repeatedly runs partial QRCP on $B$ to pick $b$ pivots, computes the QR factorization on the pivoted columns, forms a block of Householder reflections, applies them to the trailing matrix of $A$ with Level 3 BLAS, and then updates the remaining columns of $B$. RQRCP exits this process when it reaches a target rank $k$.

The RQRCP algorithm as described in Algorithm \ref{Algorithm_Partial QRCP with random sampling (RQRCP)} computes a low-rank approximation with a target rank $k$. While it can be modified to compute a low-rank approximation that satisfies an error tolerance, our analysis will be on Algorithm \ref{Algorithm_Partial QRCP with random sampling (RQRCP)}. The block size $b$ is a machine-dependent parameter experimentally chosen to optimize performance; whereas oversampling size $p$ is chosen to ensure the reliability of Algorithm \ref{Algorithm_Partial QRCP with random sampling (RQRCP)}. In practice, a value of $p$ between $5$ and $20$ suffices. At the end of each loop, matrix $B$ is updated, via one of several updating formulas, to become a random projection matrix for the trailing matrix. In Section \ref{subsection: Probability analysis of RQRCP} we will justify Algorithm \ref{Algorithm_Partial QRCP with random sampling (RQRCP)} with a rigorous probability analysis. 

\subsection{Updating formulas for B in RQRCP}
We discuss updating formulas for $B$ and their efficiency differences. Consider partial QRs on $A \Pi$ and $B \Pi = \Omega A \Pi$,
\begin{equation}\label{Eqn:QRAB}
A \Pi 
=
Q
\left(
\begin{array}{cc}
R_{11} & R_{12} \\
& R_{22}
\end{array}
\right), \quad
B \Pi = 
\widehat{Q}
\left(
\begin{array}{cc}
\widehat{R}_{11} & \widehat{R}_{12} \\
& \widehat{R}_{22}
\end{array}
\right),
\end{equation}
where $A \in \mathbb{R}^{m \times n}$. We describe two different updating formulas introduced in \cite{duersch2015true} and \cite{martinsson2017householder}, respectively.

For the first updating formula \cite{duersch2015true}, partition $\widehat{\Omega} \stackrel{def}{=} \widehat{Q}^T \Omega Q \stackrel{def}{=} \left(
\begin{array}{cc}
\widehat{\Omega}_1 & \widehat{\Omega}_2
\end{array}
\right)$,
equation \eqref{Eqn:QRAB} implies
\begin{equation}\label{Eqn:hatomega}
\widehat{\Omega} 
\left(
\begin{array}{cc}
R_{11} & R_{12} \\
& R_{22}
\end{array}
\right)
=
\left(
\begin{array}{cc}
\widehat{R}_{11} & \widehat{R}_{12} \\
& \widehat{R}_{22}
\end{array}
\right), 
\end{equation}
\noindent leading to the updating formula of Algorithm \ref{Algorithm_Partial QRCP with random sampling (RQRCP)}, 
\begin{equation}\label{Eqn:updating_1}
\left(
\begin{array}{c}
\widehat{R}_{12} \\
\widehat{R}_{22}
\end{array}
\right)
\leftarrow \widehat{\Omega}_2 R_{22}
=
\left(
\begin{array}{c}
\widehat{R}_{12} - \widehat{R}_{11} R_{11}^{-1} R_{12} \\
\widehat{R}_{22}
\end{array}
\right),
\end{equation}
\noindent where only the upper part of updating formula \eqref{Eqn:updating_1} requires computation, at a total flop cost of $O\left(n\,k\,b\right)$.

For the second updating formula \cite{martinsson2017householder}, partition  
\[ \overline{\Omega} \stackrel{def}{=} \Omega\, Q \stackrel{def}{=}  
\left(
\begin{array}{cc}
\overline{\Omega}_1 & \overline{\Omega}_2
\end{array}
\right), \; \mbox{and} \; B\, \Pi \stackrel{def}{=} \left(
\begin{array}{cc}
B_1 & B_2
\end{array}
\right), \]
leading to an updating formula,
\begin{equation}\label{Eqn:updating_2}
\left(
\begin{array}{c}
\widehat{R}_{12} \\
\widehat{R}_{22}
\end{array}
\right)
\leftarrow \overline{\Omega}_2 R_{22} 
=
B_2 - \overline{\Omega}_1 R_{12}.
\end{equation}
\noindent The approach in \cite{martinsson2017householder} decomposes $Q$ and is mathematically equivalent to but computationally less efficient than \eqref{Eqn:updating_2}, which requires totaling $O\left(\left(b+p\right)\left(m+n\right)k\right)$ flops.

Both updating formulas are numerically stable in practice. Since $\widehat{\Omega} = \widehat{Q}^T \, \overline{\Omega}$, the updating formula \eqref{Eqn:updating_2} differs from formula \eqref{Eqn:updating_1} only by a pre-multiplied orthogonal matrix $\widehat{Q}.$  Consequently, both updating formulas produce identical permutations and thus identical low-rank approximations. However, updating formula \eqref{Eqn:updating_1} requires $50\%$ fewer flops than \eqref{Eqn:updating_2} and is much faster in numerical experiments. 

\subsection{Probability analysis of RQRCP}\label{subsection: Probability analysis of RQRCP}
In this section, we show that RQRCP is as reliable as QRCP up to negligible failure probabilities. Since QRCP chooses the column with the largest norm as the pivot column in each loop, QRCP satisfies the following property. 
\begin{lem}[QRCP pseudo-diagonal dominance]
\label{Thm:pseudo-diagonal dominance of QRCP}
Let $A \in \mathbb{R}^{m \times n}$. Perform a $k$-step partial QRCP on $A$,
\begin{equation*}
A\Pi=QR=Q\left(
\begin{array}{cc}
R_{11} & R_{12} \\
& R_{22} 
\end{array}
\right), \quad \mbox{with} \quad R = \left(r_{ij}\right),
\end{equation*}
\begin{equation*}
\mbox{then} \, \left|r_{ii}\right| \ge \sqrt{\sum_{l=i}^{m} \left|r_{lj}\right|^2} \, (1 \le i \le k, i+1 \le j \le n). 
\end{equation*}
\end{lem}
A similar property holds for RQRCP with target rank $k$.
\begin{thm}[RQRCP pseudo-diagonal dominance]\label{Thm:reliability of RQRCP}
$A \in \mathbb{R}^{m \times n}, \varepsilon, \Delta \in (0,1)$. Draw $\Omega \in \mathcal{N}(0,1)^{\left(b+p\right)\times m}$ with $p \ge \lceil \frac{4}{\varepsilon^2-\varepsilon^3} \, {\bf log} \frac{2nk}{\Delta}\rceil -1$. Perform RQRCP on $A$ given $k$, 
\begin{equation*}
A\Pi=QR=Q\left(
\begin{array}{cc}
R_{11} & R_{12} \\
& R_{22} \\
\end{array}
\right), \quad \mbox{with} \quad R = \left(r_{ij}\right),
\end{equation*}
\begin{equation*}
\mbox{then} \, \left|r_{ii}\right| \ge \sqrt{\frac{1-\varepsilon}{1+\varepsilon}} \sqrt{\sum_{l=i}^{m} \left|r_{lj}\right|^2} \,
(1 \le i \le k, i+1 \le j \le n) 
\end{equation*}
with probability at least $1 - \Delta$.
\end{thm}

We prove Theorem \ref{Thm:reliability of RQRCP} in two stages. Firstly, we prove that the updating formulas for $B$ are indeed products of i.i.d. Gaussian matrices and the trailing matrix of $A$, conditional on $\Pi$. Secondly, we use Johnson-Lindenstrauss Theorem \cite{dasgupta2003elementary}  and the law of total probability to establish Theorem \ref{Thm:reliability of RQRCP}. 

\subsubsection{Correctness of Updating formulas for B}
For correctness, we show that the matrices $\widehat{\Omega}_2$ and $\overline{\Omega}_2$ in updating formulas \eqref{Eqn:updating_1} and \eqref{Eqn:updating_2} remain i.i.d. Gaussian given permutation matrix $\Pi$. Consider $s$-step partial QRs of $A \Pi$ and $B \Pi$ in equation \eqref{Eqn:QRAB} with $R_{11}, \widehat{R}_{11} \in \mathbb{R}^{s \times s}$. Recall that
\[\overline{\Omega} = \Omega \, Q = \left(
\begin{array}{cc}
\overline{\Omega}_1 & \overline{\Omega}_2
\end{array}\right), \quad \widehat{\Omega} = \widehat{Q}^T \overline{\Omega} = 
\left(
\begin{array}{cc}
\widehat{\Omega}_1 & \widehat{\Omega}_2
\end{array}
\right).\]

Given $\Pi$, $Q$ is an orthogonal matrix decorrelated with $\Omega$ so $\overline{\Omega} = \Omega Q$ remains i.i.d. Gaussian. From equation \eqref{Eqn:QRAB},
\begin{equation*}
\widehat{Q}\left(
\begin{array}{cc}
\widehat{R}_{11} & \widehat{R}_{12} \\
& \widehat{R}_{22} \\
\end{array}
\right)=\overline{\Omega}\left(
\begin{array}{cc}
R_{11} & R_{12} \\
& R_{22} \\
\end{array}
\right).
\end{equation*}
{\noindent}The orthogonal matrix $\widehat{Q}$ is completely determined by its first $s$ columns, which in turn are completely determined by $\overline{\Omega}_1$, which is decorrelated with $\overline{\Omega}_2$. 

It now follows that $\widehat{\Omega}_2 = \widehat{Q}^T\,\overline{\Omega}_2 $ must remain i.i.d. Gaussian given $\Pi$. The matrices $\widehat{\Omega}_2$ and $\overline{\Omega}_2$ in updating formulas \eqref{Eqn:updating_1} and \eqref{Eqn:updating_2} correspond to the special case $s = b$, and thus must remain i.i.d. given permutation matrix $\Pi$.

Applying above argument on each loop in RQRCP, we conclude that, with both updating formulas, the remaining columns of $B$ in each loop are always a product of an i.i.d. Gaussian matrix and the trailing matrix of $A$, conditional on $\Pi$. 

Additionally, equation \eqref{Eqn:hatomega} implies 
that $\widehat{\Omega}$ must be a $2\times 2$ block upper triangular  
matrix since $R_{11}$ is invertible for any $0 \le s \le b-1$: 
\[\widehat{\Omega} = \left(
\begin{array}{cc}
\widehat{\Omega}_1 & \widehat{\Omega}_2
\end{array}\right) = \left(
\begin{array}{cc}
\widehat{\Omega}_{11} & \widehat{\Omega}_{12} \\
 & \widehat{\Omega}_{22} 
\end{array}\right), \] 
which allows us to write from equation \eqref{Eqn:hatomega}
\[ \widehat{R}_{22} = \widehat{\Omega}_{22} \, {R}_{22}. \]
Thus every trailing matrix in $B$ is a matrix product of an i.i.d. Gaussian matrix with $b+p-s \ge p + 1$ rows and the trailing matrix of $A$, given $\Pi$. Together with above argument on updating formulas, we conclude that {\em every} pivot computed by Algorithm \ref{Algorithm_Partial QRCP with random sampling (RQRCP)} is based on choosing the column with the largest column norm of the matrix product of an i.i.d. Gaussian matrix and a trailing matrix of $A$, conditional on the corresponding column permutation matrix $\Pi$. 

\subsubsection{Probability analysis of reliability of RQRCP}
Our probability analysis is based on the Johnson-Lindenstrauss Theorem \cite{dasgupta2003elementary}. We say a vector $x\in \mathbb{R}^{d}$ satisfies the Johnson$-$Lindenstrauss condition for given $\varepsilon > 0$ under i.i.d Gaussian matrix $\Omega \in \mathcal{N}(0,1)^{r \times d}$ if 
\begin{equation*}
\left(1-\varepsilon\right)\left\|x\right\|_2^2 \le \left|\left|\frac{1}{\sqrt{r}} \Omega x \right|\right|_2^2 \le \left(1+\varepsilon\right)\left\|x\right\|_2^2.
\end{equation*}
{\noindent}The Johnson$-$Lindenstrauss Theorem states that a vector $x$ satisfies the Johnson$-$Lindenstrauss condition with high probability 
\begin{eqnarray*}
&& P\left((1-\varepsilon)\left\|x\right\|_2^2\le\left|\left|\frac{1}{\sqrt{r}}\Omega x\right|\right|_2^2\le(1+\varepsilon)\left\|x\right\|_2^2\right)\\
&& \geq  1-2\exp\left(-\frac{\left(\varepsilon^2-\varepsilon^3\right)r}{4}\right).
\end{eqnarray*}

\normalsize
\noindent Before we prove Theorem \ref{Thm:reliability of RQRCP}, we need Lemma \ref{lemma: reliability of RQRCP without updating B}.
\begin{lem}\label{lemma: reliability of RQRCP without updating B} Suppose that RQRCP computes an $s$-step partial QRCP factorization $A\Pi = Q R$ with $R = \left(r_{i,j}\right)$, then for any $\varepsilon \in \left(0,1\right)$,
\begin{equation}\label{Eqn:rss}
\left|r_{s,s}\right| \ge \sqrt{\frac{1-\varepsilon}{1+\varepsilon}} \sqrt{\sum_{l=s}^{m} \left|r_{lj}\right|^2}, \quad\left(s+1 \le j \le n\right)
\end{equation}
{\noindent}with probability at least $1 - \Delta_s$ where
\begin{equation*}
\Delta_s=2(n-s+1)\exp\left(\frac{-(\varepsilon^2-\varepsilon^3)(p+1)}{4}\right).
\end{equation*}
\end{lem}

\begin{proof}[Proof of Lemma \ref{lemma: reliability of RQRCP without updating B}] Let $\Pi$ be a random permutation in equation \eqref{Eqn:QRAB} with $R_{11},\,\widehat{R}_{11} \in \mathbb{R}^{(s-1) \times (s-1)}$, let $E_{sj}$ be the event where  $R_{22}\left(:,j\right)$ satisfies the Johnson$-$Lindenstrauss condition. Therefore $E_s = \bigcap_{j=1}^{n-s+1}E_{sj}$ is the event where all columns of $R_{22}$ satisfy the Johnson$-$Lindenstrauss condition. By correctness of RQRCP, the event $E_{sj}$ for any given $\Pi$ satisfies the Johnson$-$Lindenstrauss condition, therefore
\begin{eqnarray*}
&& P\left(E_s|\Pi\right) \ge \sum_{j=1}^{n-s+1} P\left(E_{sj}|\Pi\right) - \left(n-s\right) \\
&\ge& \!\!\!\!\!\!\!\!\sum_{j=1}^{n-s+1} \left(1 - 2\exp\left(\frac{-\left(\varepsilon^2-\varepsilon^3\right)\left(p+1\right)}{4}\right)\right) - \left(n-s\right) \\
&=& \hspace{-0.2cm} 1 - \Delta_s.
\end{eqnarray*}
We now remove the condition $\Pi$ by law of total probability, 
\begin{equation*}
P\left(E_s\right)=\sum_{\Pi}P\left(E_s|\Pi\right)P\left(\Pi\right) \ge 1 - \Delta_s.
\end{equation*}
\noindent Denote columns of $R_{22}$ as $r_{s}, r_{s+1}, \dots, r_n$, columns of $\widehat{R}_{22}$ as  $\widehat{r}_{s}, \widehat{r}_{s+1}$, $ \dots, \widehat{r}_n$ {\em after} the $s^{{\mbox{th}}}$ column pivot, then $\|\widehat{r}_{s}\|_2 \ge \|\widehat{r}_j\|_2$ for $s+1 \le j \le n.$
With probability at least $1 - \Delta_s$,  event $E_s$ holds so that 
\begin{equation*}
\frac{\|r_j\|_2}{\|r_{s}\|_2} \le \frac{\sqrt{1+\varepsilon}\|\widehat{r}_j\|_2}{\sqrt{1-\varepsilon}\|\widehat{r}_{s}\|_2} \le \sqrt{\frac{1+\varepsilon}{1-\varepsilon}},\; \mbox{for all $s+1 \le j \le n$,} 
\end{equation*}
\noindent equivalent to inequality \eqref{Eqn:rss} after one-step QR.  
\end{proof}

\begin{proof}[Proof of Theorem \ref{Thm:reliability of RQRCP}]Denote $E$ as the event where all column lengths used in comparison satisfy the Johnson$-$Lindenstrauss condition so that $E = \bigcap_{s=1}^{k}E_s$, with $E_s$ being the event where all columns of $R_{22}$ satisfy the Johnson$-$Lindenstrauss condition in $s^{\mbox{th}}$ step. By Lemma \ref{lemma: reliability of RQRCP without updating B}, 
\begin{eqnarray*}
P\left(E\right)\!\!\!\!\!&\ge&\!\!\!\!\!\sum_{s=1}^{k} P\left(E_s\right) - \left(k-1\right) \ge \sum_{s=1}^{k} \left(1-\Delta_{s}\right) - \left(k-1\right) \\
&=&\!\!\!\!\! 1 - \sum_{s=1}^{k} \Delta_{s}\\
&=&\!\!\!\!\! 1 - \sum_{s=1}^{k} 2\left(n-s+1\right)\exp\left(\frac{-\left(\varepsilon^2-\varepsilon^3\right)\left(p+1\right)}{4}\right) \\
&\ge&\!\!\!\!\! 1 - 2n\,k\exp\left(\frac{-\left(\varepsilon^2-\varepsilon^3\right)\left(p+1\right)}{4}\right),
\end{eqnarray*}
which is at least $1-\Delta$ for the choice of $p$ in Theorem \ref{Thm:reliability of RQRCP}.
\end{proof}

Theorem \ref{Thm:reliability of RQRCP} suggests that $p$ needs only to grow logarithmically with $n$, $k$ and $1/\Delta$ for reliable column selection. For illustration, if we choose $n = 1000, k = 200, \varepsilon = 0.5, \Delta = 0.05$, then $p \ge 508$. However, in practice, $p$ values like $5 \sim 20$ suffice for RQRCP to obtain high quality low-rank approximations.

\section{Spectrum-revealing QR Factorization}
\subsection{Introduction}
Before we introduce spectrum-revealing QR factorization (SRQR), we need Lemma \ref{lemma:singular value inequality} for partial QR factorizations with column interchanges. Let $\sigma_j(X)$ be the $j^{\mbox{th}}$ largest singular value of any matrix $X$, and let $\lambda_j\left(H\right)$ be the $j^{\mbox{th}}$ largest eigenvalue of any positive semidefinite matrix $H$. 

\begin{lem}\label{lemma:singular value inequality}
Let $A \in \mathbb{R}^{m \times n}$, with $1 \le l \le \min(m,n)$. For any permutation $\Pi$, consider a block QR factorization
\begin{equation*}
A \Pi = Q \left(
\begin{array}{cc}
R_{11} & R_{12} \\
& R_{22}
\end{array}
\right),
\end{equation*}
\noindent with $R_{11} \in \mathbb{R}^{l \times l}$. Define 
$
\widetilde{R} = \left(
R_{11} \; R_{12}
\right).
$
For any $1 \le k \le l$, denote $\widetilde{R}_{k}$ the rank-$k$ truncated SVD of $\widetilde{R}$. Therefore, 
\begin{equation}\label{Eqn:singular value of A and R}
\sigma_j^2\left(A\right) \le \sigma_j^2\left(\widetilde{R}\right) + \left\|R_{22}\right\|_2^2, \quad (1 \le j \le k),
\end{equation}
\begin{equation}\label{Eqn:2norm of truncated QR}
\left\|A\Pi - Q\left(
\begin{array}{c}
\widetilde{R}_k \\
0
\end{array}
\right)\right\|_2 \le \sigma_{k+1}\left(A\right)\sqrt{1 + \left(\frac{\left\|R_{22}\right\|_2}{\sigma_{k+1}\left(A\right)}\right)^2}.
\end{equation}
\end{lem}

We introduced a new parameter $l$, in Lemma \ref{lemma:singular value inequality}. It will become clear later on that an $l$-step partial QR with properly chosen $\Pi$ for an $l$ somewhat larger than $k$ can lead to a much better rank-$k$ approximation than a $k$-step partial QR. 

\begin{proof}[Proof of Lemma \ref{lemma:singular value inequality}]
By definition, for $1 \le j \le k$, 
\begin{eqnarray*}
\sigma_j^2\left(A\right) &=& \lambda_j\left(\Pi^TA^TA\Pi\right) \\
&=& 
\lambda_j\left(\widetilde{R}^T \widetilde{R} + \left(
\begin{array}{cc}
0 & 0 \\
0 & R_{22}^TR_{22}
\end{array}
\right)\right) \\
&\le& \lambda_j\left(\widetilde{R}^T\widetilde{R}\right) + \left\|R_{22}^TR_{22}\right\|_2 \\
&=& \sigma_j^2\left(\widetilde{R}\right) + \left\|R_{22}\right\|_2^2.
\end{eqnarray*}
\begin{eqnarray*}
&&\left\|A\Pi - Q\left(
\begin{array}{c}
\widetilde{R}_k \\
0
\end{array}
\right)\right\|_2^2 
= \left\|Q\left(
\begin{array}{c}
\widetilde{R} - \widetilde{R}_k \\
\overline{R}
\end{array}
\right)\right\|_2^2 \\
&\le& \left\|\widetilde{R}-\widetilde{R}_k\right\|_2^2 + \left\|\overline{R}\right\|_2^2 \le \sigma_{k+1}^2(A) + \left\|R_{22}\right\|_2^2, 
\end{eqnarray*}
where $\overline{R} \stackrel{def}{=} \left(
\begin{array}{cc}
0 & R_{22}
\end{array}
\right)$.
\end{proof}

\subsection{Spectrum-revealing QR}
We exhibit the properties of equation \eqref{Eqn:singular value of A and R} in more detail.
\begin{eqnarray*}
\sigma_j^2(A) &\le& \sigma_j^2\left(\widetilde{R}\right)\left(1 + \frac{\left\|R_{22}\right\|_2^2}{\sigma_j^2\left(\widetilde{R}\right)}\right) \\
&\le& \sigma_j^2\left(\widetilde{R}\right)\left(1 + \frac{\left\|R_{22}\right\|_2^2}{\sigma_k^2\left(\widetilde{R}\right)}\right), \quad \mbox{or} 
\end{eqnarray*}
\begin{equation}\label{Eqn:singular value inequality1}
\sigma_j\left(\widetilde{R}\right) \ge \frac{\sigma_j(A)}{\sqrt{1 + \left(\frac{\left\|R_{22}\right\|_2}{\sigma_k\left(\widetilde{R}\right)}\right)^2}}, \quad (1 \le j \le k).
\end{equation}
Additionally,
\begin{eqnarray*}
\sigma_j^2(A) \!\!\!\!&\le&\!\!\!\! \sigma_j^2\left(\widetilde{R}\right)\left(1 + \frac{\sigma_j^2(A)}{\sigma_j^2\left(\widetilde{R}\right)}\frac{\left\|R_{22}\right\|_2^2}{\sigma_j^2(A)}\right) \\
&\le&\!\!\!\! \sigma_j^2\left(\widetilde{R}\right)\left(1 + \frac{\left(\sigma_j^2\left(\widetilde{R}\right) + \left\|R_{22}\right\|_2^2\right)}{\sigma_j^2\left(\widetilde{R}\right)} \frac{\left\|R_{22}\right\|_2^2}{\sigma_j^2(A)}\right) \\
&\le&\!\!\!\! \sigma_j^2\left(\widetilde{R}\right) \left(1 + \left(1 + \frac{\left\|R_{22}\right\|_2^2}{\sigma_k^2\left(\widetilde{R}\right)}\right)\frac{\left\|R_{22}\right\|_2^2}{\sigma_j^2(A)}\right), \; \mbox{or} 
\end{eqnarray*}
\begin{equation}\label{Eqn:singular value inequality2}
\sigma_j\left(\widetilde{R}\right) \ge \frac{\sigma_j(A)}{\sqrt{1 + \left(1 + \frac{\left\|R_{22}\right\|_2^2}{\sigma_k^2\left(\widetilde{R}\right)}\right)\frac{\left\|R_{22}\right\|_2^2}{\sigma_j^2(A)}}}, \quad (1 \le j \le k).
\end{equation}

By the Cauchy interlacing property,
\begin{equation*}
\sigma_l\left(R_{11}\right)\le\sigma_l\left(\widetilde{R}\right)\le\sigma_l(A) \quad \mbox{and} \quad \left\|R_{22}\right\|_2 \ge \sigma_{l+1}(A).
\end{equation*}
Relations \eqref{Eqn:2norm of truncated QR}\eqref{Eqn:singular value inequality1}\eqref{Eqn:singular value inequality2} show that we can reveal the leading singular values of $A$ in $\widetilde{R}$ well by computing a $\Pi$ that ensures 
\begin{equation}\label{Eqn:R22 bound condition of SRQR}
\left\|R_{22}\right\|_2 \le O(\sigma_{l+1}(A)).
\end{equation}
Indeed, equation \eqref{Eqn:singular value inequality2} further ensures that such permutation $\Pi$ would make all the leading singular values of $\widetilde{R}$ very close to those of $A$ if the singular values of $A$ decay relatively quickly. Finally, equation \eqref{Eqn:2norm of truncated QR} guarantees that such a permutation would also lead to a high quality approximation measured in $2$-norm. In summary, for a successful SRQR, we only need to find a $\Pi$ that satisfies equation \eqref{Eqn:R22 bound condition of SRQR}.

For most matrices in practice, both QRCP and RQRCP compute high quality SRQRs, with RQRCP being  significantly more efficient. We will develop an algorithm for computing an SRQR in $3$ stages:
\begin{enumerate}
  \item Compute an $l$-step QRCP or RQRCP.
  \item Verify condition \eqref{Eqn:R22 bound condition of SRQR}.
  \item Compute an SRQR if condition \eqref{Eqn:R22 bound condition of SRQR} does not hold. 
\end{enumerate}

In next section, we develop a rather efficient scheme to verify if a partial QR factorization satisfies \eqref{Eqn:R22 bound condition of SRQR}.

\subsection{SRQR verification}
Given a QRCP or RQRCP, we discuss an efficient scheme to check whether it satisfies condition \eqref{Eqn:R22 bound condition of SRQR}. We define
\begin{equation}\label{Eqn:definition of g1 and g2}
g_1 \stackrel{def}{=}  \frac{\left\|R_{22}\right\|_{1,2}}{\left|\alpha\right|} \quad \mbox{and} \quad g_2 \stackrel{def}{=}  \left|\alpha\right| \left\|\widehat{R}^{-T}\right\|_{1,2},
\end{equation}
where $\left\|X\right\|_{1,2}$ is the largest column norm of any $X$ and 
\begin{equation*}
\widehat{R} \stackrel{def}{=} \left(
\begin{array}{cc}
R_{11} & a \\
& \alpha
\end{array}
\right) \;\; {\mbox{is a leading submatrix of $R$,}} 
\end{equation*}
\noindent where we do one step of QRCP or RQRCP on $R_{22}$. Then
\begin{eqnarray}
\left\|R_{22}\right\|_2 &=& \frac{\left\|R_{22}\right\|_2}{\left\|R_{22}\right\|_{1,2}}\left\|R_{22}\right\|_{1,2}  = \tau \sigma_{l+1}\left(A\right),\label{Eqn:inequality of tau} \\
\mbox{where} \quad \tau &\stackrel{def}{=}& g_1 g_2 \frac{\left\|R_{22}\right\|_2}{\left\|R_{22}\right\|_{1,2}} \frac{\left\|\widehat{R}^{-T}\right\|_{1,2}^{-1}}{\sigma_{l+1}\left(A\right)}.\nonumber 
\end{eqnarray}
\noindent While $\tau$ depends on $A$ and $\Pi$, it can be upper bounded as 
\begin{eqnarray}\label{Eqn:upper bound of tau}
\tau &=& g_1 g_2 \frac{\left\|R_{22}\right\|_2}{\left\|R_{22}\right\|_{1,2}} \frac{\left\|\widehat{R}^{-T}\right\|_{1,2}^{-1}}{\sigma_{l+1}\left(\widehat{R}\right)} \frac{\sigma_{l+1}\left(\widehat{R}\right)}{\sigma_{l+1}\left(A\right)} \\
&\le& g_1 g_2 \sqrt{(l+1)(n-l)}. \nonumber
\end{eqnarray}
On the other hand, 
\begin{eqnarray}
\left\|R_{22}\right\|_2 &=& \frac{\left\|R_{22}\right\|_2}{\left\|R_{22}\right\|_{1,2}}\left\|R_{22}\right\|_{1,2} \le  \widehat{\tau} \sigma_{l}\left(\widetilde{R}\right),\label{Eqn:inequality of widehat}\\
\mbox{where} \quad \widehat{\tau} &\stackrel{def}{=}& g_1 g_2 \frac{\left\|R_{22}\right\|_{2}}{\left\|R_{22}\right\|_{1,2}} \frac{\left\|R_{11}^{-T}\right\|_{1,2}^{-1}}{\sigma_l\left(\widetilde{R}\right)}.\nonumber 
\end{eqnarray}
\noindent While $\widehat{\tau}$ depends on $A$ and $\Pi$, it can be upper bounded as
\begin{equation}\label{Eqn:upper bound of widehat tau}
\widehat{\tau} = g_1 g_2 \frac{\left\|R_{22}\right\|_{2}}{\left\|R_{22}\right\|_{1,2}} \frac{\left\|R_{11}^{-T}\right\|_{1,2}^{-1}}{\sigma_l\left(R_{11}\right)} \frac{\sigma_l\left(R_{11}\right)}{\sigma_l\left(\widetilde{R}\right)} \le g_1 g_2 \sqrt{l(n-l)} .
\end{equation}
\noindent It is typical for the first two ratios in equations \eqref{Eqn:upper bound of tau} and \eqref{Eqn:upper bound of widehat tau} to be of order $O(1)$, and the last ratio $o(1)$. Thus, even though the last upper bounds in equations \eqref{Eqn:upper bound of tau} and \eqref{Eqn:upper bound of widehat tau} grow with matrix dimensions, $\tau$ is small to modest in practice.

For notational simplicity, we further define $\overline{\tau} \stackrel{def}{=} \widehat{\tau}\frac{\sigma_l\left(\widetilde{R}\right)}{\sigma_k\left(\widetilde{R}\right)}$. Plugging equation \eqref{Eqn:inequality of widehat} into equation \eqref{Eqn:singular value inequality1}, we obtain
\begin{equation}\label{Eqn:singular value inequality1 with tau}
\sigma_j\left(\widetilde{R}\right) \ge \frac{\sigma_j(A)}{\sqrt{1 + \overline{\tau}^2}}, \quad (1 \le j \le k).
\end{equation}
\noindent Plugging both equations \eqref{Eqn:inequality of tau} and \eqref{Eqn:inequality of widehat} into equation \eqref{Eqn:singular value inequality2},
\begin{equation*}
\sigma_j\left(\widetilde{R}\right) \ge \frac{\sigma_j(A)}{\sqrt{1 + \tau^2 \left(1 + \overline{\tau}^2\right) \left(\frac{\sigma_{l+1}(A)}{\sigma_j(A)}\right)^2}}, \quad (1 \le j \le k).
\end{equation*}
\noindent Combining the last equation with \eqref{Eqn:singular value inequality1 with tau}, for $1 \le j \le k$, 
\begin{equation}\label{Eqn:singular value inequality2 with tau}
\sigma_j\left(\widetilde{R}\right) \ge \frac{\sigma_j(A)}{\sqrt{1 + \min\left(\overline{\tau}^2, \tau^2\left(1 + \overline{\tau}^2\right) \left(\frac{\sigma_{l+1}(A)}{\sigma_j(A)}\right)^2\right)}}.
\end{equation}

\noindent Equation \eqref{Eqn:singular value inequality2 with tau} shows that, under definition \eqref{Eqn:definition of g1 and g2} and with $\widetilde{R}$, we can reveal at least a dimension dependent fraction of all the leading singular values of $A$ and indeed approximate them very accurately in case they decay relatively quickly. 

Finally, plugging equation \eqref{Eqn:inequality of tau} into equation \eqref{Eqn:2norm of truncated QR},
\begin{equation}\label{Eqn:2norm bound with tau}
\left\|A\Pi - Q\left(
\begin{array}{c}
\widetilde{R}_k \\
0
\end{array}
\right)\right\|_2 \le
\sigma_{k+1}(A) \sqrt{1 + \tau^2 \left(
\frac{\sigma_{l+1}(A)}{\sigma_{k+1}(A)}
\right)^2}.
\end{equation}
\noindent Equation \eqref{Eqn:2norm bound with tau} shows that we can compute a rank-$k$ approximation that is up to a factor of $\sqrt{1 + \tau^2 \left(
\frac{\sigma_{l+1}(A)}{\sigma_{k+1}(A)}
\right)^2}$ from optimal. In situations where singular values of $A$ decay relatively quickly, our rank-$k$ approximation is about as accurate as the truncated SVD with a choice of $l$ such that 
\begin{equation*}
\frac{\sigma_{l+1}(A)}{\sigma_{k+1}(A)} = o(1).
\end{equation*}

\subsection{Spectrum-revealing bounds of QRCP and RQRCP}
We develop spectrum-revealing bounds of QRCP and RQRCP. It comes down to estimating upper bounds on $g_1$ and $g_2$. We need Lemma \ref{lemma: inverse 1}, presented below without a proof:
\begin{lem}\label{lemma: inverse 1}
Let $W=(w_{i,j}) \in \mathbb{R}^{n \times n}$ be an upper or lower triangular matrix with $w_{i,i}=1$ and $\left|w_{ij}\right| \le c \;(i \neq j)$. Then 
\begin{equation*}
\left\|W^{-1}\right\|_{1,2}\le\left\|W^{-1}\right\|_1\le(1+c)^{n-1}.
\end{equation*}
\end{lem}

For QRCP, $g_1 = 1$ since $\left|\alpha\right|$ is the largest column norm in the trailing matrix $R_{22}$. For $g_2$, decompose $\widehat{R} = DW$ where $D$ is the diagonal of $\widehat{R}$ and $W$ satisfies Lemma \ref{lemma: inverse 1} with $c=1$.
\begin{equation*}
g_2 = \left|\alpha\right| \left\|D^{-T} W^{-T}\right\|_{1,2} \le \left\|W^{-T}\right\|_{1,2} \le \left\|W^{-T}\right\|_1
\le 2^l.
\end{equation*}
For RQRCP, we find upper bounds for $g_1$ and $g_2$. According to our probability analysis of RQRCP, the following inequalities are valid with probability $1 - \Delta$ for given $\varepsilon$.
\begin{eqnarray*}
|\alpha| &\ge& \sqrt{\frac{1-\varepsilon}{1+\varepsilon}}\left\|R_{22}\right\|_{1,2} \Rightarrow g_1 = \frac{\left\|R_{22}\right\|_{1,2}}{|\alpha|} \le \sqrt{\frac{1+\varepsilon}{1-\varepsilon}}. \\
g_2^2 
&=& \alpha^2 \left|\left|\left(
\begin{array}{cc}
R_{11}^{-T} &  \\
-\frac{1}{\alpha}a^TR_{11}^{-T}  & \frac{1}{\alpha} \\
\end{array}
\right)\right|\right|_{1,2}^2 \\
&\le& \alpha^2 \max \bigg\{\left|\left|R_{11}^{-T}\right|\right|_{1,2}^2 + \left|\left|-\frac{1}{\alpha}a^TR_{11}^{-T}\right|\right|_{1,2}^2, \frac{1}{\alpha^2}\bigg\}.
\end{eqnarray*}

We decompose $R_{11} = DW$ where $D$ is the diagonal of $R_{11}$ and $W$ satisfies Lemma \ref{lemma: inverse 1} with $c=\sqrt{\frac{1+\varepsilon}{1-\varepsilon}}$. Assume the smallest entry in $D$ is $t$, then $t$ satisfies $t \ge \sqrt{\frac{1-\varepsilon}{1+\varepsilon}} \sigma_l(A)$,
\begin{eqnarray*}
&&\hspace{-0.3cm}\left\|R_{11}^{-T}\right\|_{1,2}^2
= \left\|D^{-1}\left(W^T\right)^{-1}\right\|_{1,2}^2
\le \frac{1}{t^2} \left\|\left(W^T\right)^{-1}\right\|_{1,2}^2 \\
&\le& \frac{1}{\sigma_l^2(A)} \cdot \left(\frac{1+\varepsilon}{1-\varepsilon}\right) \cdot \left(1+\sqrt{\frac{1+\varepsilon}{1-\varepsilon}}\right)^{2l-2}, \quad \mbox{and} \\
&&\left|\left|-\frac{1}{\alpha}a^TR_{11}^{-T}\right|\right|_{1,2}^2 = \frac{1}{\alpha^2} \left|\left|(D^{-1}a)^T \left(W^T\right)^{-1}\right|\right|_{1,2}^2 \\
&\le& \frac{1}{\alpha^2} \cdot \left(\frac{1+\varepsilon}{1-\varepsilon}\right) \cdot \left(1+\sqrt{\frac{1+\varepsilon}{1-\varepsilon}}\right)^{2l-2}.
\end{eqnarray*}
\noindent It follows that 
\begin{eqnarray*}
g_2^2 &\le& \frac{1+\varepsilon}{1-\varepsilon} \cdot \left(1+\sqrt{\frac{1+\varepsilon}{1-\varepsilon}}\right)^{2l-2} \cdot \left(\frac{\alpha^2}{\sigma_l^2(A)}+1\right) \\
&\le& \frac{1+\varepsilon}{1-\varepsilon} \cdot \left(1+\sqrt{\frac{1+\varepsilon}{1-\varepsilon}}\right)^{2l-2} \cdot \left(\frac{1+\varepsilon}{1-\varepsilon}+1\right) \\
&\Rightarrow& g_2 \le \frac{\sqrt{2(1+\varepsilon)}}{1-\varepsilon} \left(1+\sqrt{\frac{1+\varepsilon}{1-\varepsilon}}\right)^{l-1}. 
\end{eqnarray*} 
In summary, 
\begin{eqnarray}
g_1 &\le & {\displaystyle \left\{\begin{array}{ll}
    1 & \mbox{for QRCP}, \\
  \sqrt{\frac{1+\varepsilon}{1-\varepsilon}}   &  \mbox{for RQRCP}.
\end{array}\right. } \label{Eqn:g1}\\
\; g_2 &\le & {\displaystyle \left\{\begin{array}{ll}
    2^l &  \mbox{for QRCP}, \\
  \frac{\sqrt{2(1+\varepsilon)}}{1-\varepsilon} \left(1+\sqrt{\frac{1+\varepsilon}{1-\varepsilon}}\right)^{l-1}   & \mbox{for RQRCP}.
\end{array}\right.} \label{Eqn:g2}
\end{eqnarray}

\subsection{An algorithm to compute SRQR}
\begin{algorithm}
\caption{Spectrum-revealing QR factorization (SRQR)}\label{algorithm_srqr}
\begin{algorithmic}
\STATE $\textbf{Inputs:}$
\STATE $A$ is $m \times n$ matrix, 
\STATE $l \ge k$, the target rank, $1 \le k \le \min\left(m,n\right)$
\STATE $g > 1$ is user defined tolerance for $g_2$
\STATE $\textbf{Outputs:}$
\STATE $Q$ is $m \times m$ orthogonal matrix
\STATE $R$ is $m \times n$ upper trapezoidal matrix
\STATE $\Pi$ is $n \times n$ permutation matrix such that $A\Pi = QR$
\STATE $\textbf{Algorithm:\,\,}$ 
\STATE Compute $Q, R, \Pi$ with Algorithm \ref{Algorithm_Partial QRCP with random sampling (RQRCP)}
\STATE Compute squared 2-norm of the columns of $B(:,l+1:n): \widehat{r}_i \; (l+1 \le i \le n)$
\STATE Approximate squared 2-norm of the columns of $A(l+1:m,l+1:n): r_i = \frac{\widehat{r}_i}{b+p} \; (l+1 \le i \le n)$
\STATE $\imath ={\bf argmax}_{l+1 \le i \le n}\{r_i\}$
\STATE Swap $\imath$-th and $(l+1)$-st columns of $A, \Pi, r$
\STATE One-step QR factorization of $A(l+1:m,l+1:n)$
\STATE $\left|\alpha\right| = R_{l+1,l+1}$
\STATE $r_i = r_i - A(l+1,i)^2 \; (l+2 \le i \le n)$ 
\STATE Generate random matrix $\Omega \in\mathcal{N}(0,1)^{d\times(l+1)} \; (d \ll l)$ 
\STATE Compute $g_2 = \left|\alpha\right| \left\|\widehat{R}^{-T}\right\|_{1,2} \approx \frac{\left|\alpha\right|}{\sqrt{d}} \left\|\Omega \widehat{R}^{-T}\right\|_{1,2}$
\STATE \WHILE{$g_2 > g$}
\STATE $\imath ={\bf argmax}_{1 \le i \le l+1}\{i\text{th column norm of } \Omega \widehat{R}^{-T}\}$
\STATE Swap $\imath$-th and $(l+1)$-st columns of $A$ and $\Pi$ in a Round Robin rotation
\STATE Givens-rotate $R$ back into upper-trapezoidal form
\STATE $r_{l+1} = R_{l+1,l+1}^2$
\STATE $r_i = r_i + A(l+1,i)^2 \; (l+2 \le i \le n)$ 
\STATE $\imath ={\bf argmax}_{l+1 \le i \le n}\{r_i\}$
\STATE Swap $\imath$-th and $(l+1)$-st columns of $A, \Pi, r$
\STATE One-step QR factorization of $A(l+1:m,l+1:n)$
\STATE $\left|\alpha\right| = R_{l+1,l+1}$
\STATE $r_i = r_i - A(l+1,i)^2 \; (l+2 \le i \le n)$ 
\STATE Generate random matrix $\Omega \in\mathcal{N}(0,1)^{d\times(l+1)} (d \ll l)$ 
\STATE Compute $g_2 = \left|\alpha\right| \left\|\widehat{R}^{-T}\right\|_{1,2} \approx \frac{\left|\alpha\right|}{\sqrt{d}} \left\|\Omega \widehat{R}^{-T}\right\|_{1,2}$
\ENDWHILE
\end{algorithmic}
\end{algorithm}

The parameter $g_1$ is always modest in equation \eqref{Eqn:g1}. Despite the exponential nature of the upper bound on the parameter $g_2$ in equation \eqref{Eqn:g2}, $g_2$ has always been known to be modest with real data matrices. However, it can be exceptionally large for contrived pathological matrices \cite{kahan1966numerical}. This motivates Algorithm \ref{algorithm_srqr} for computing SRQR. Algorithm \ref{algorithm_srqr} computes a partial QR factorization with RQRCP. It then efficiently estimates $g_2$, and performs additional column interchanges for a guaranteed SRQR only if $g_2$ is too large. 

After RQRCP, Algorithm \ref{algorithm_srqr} uses a randomized scheme to quickly and reliably estimate $g_2$ and compares it to a user-defined tolerance $g > 1$. Algorithm \ref{algorithm_srqr} exits if the estimated $g_2 \le g$. Otherwise, it swaps the $\imath^{\mbox{th}}$ and $(l+1)^{\mbox{st}}$ columns of $A$ and $R$. Each swap will make the $\imath^{\mbox{th}}$ column out of the upper-triangular form in $R$. A round robin rotation is applied to the columns of $A$ and $R$, followed by a quick sequence of Givens rotations left multiplied to $R$ to restore its upper-triangular form. The while loop in Algorithm \ref{algorithm_srqr} will stop, after a finite number of swaps, leading to a permutation that ensures  $g_2 \le g$ with high probability.

The cost of estimating $g_2$ is $O(d\,l^2)$, and the cost for one extra swap is $O(n\,l)$. Matrix $R$ can be SVD-compressed into a rank-$k$ matrix optionally, at cost of $O(n\,l^2)$. There is no practical need for extra swaps for real data matrices, making Algorithm \ref{algorithm_srqr} only slightly more expensive than RQRCP in general. At most a few swaps are enough, adding an insignificant amount of computation to the cost of RQRCP, for pathological matrices like the Kahan matrix \cite{kahan1966numerical}. 

\section{Experimental performance}
Experiments in sections \ref{subsection_Approximation quality comparison on datasets}, \ref{subsection_Comparison on a pathological matrix: the Kahan matrix} and \ref{subsection_SRQR based CUR and CX matrix decomposition algorithms} were performed on a laptop with 2.7 GHz Intel Core i5 CPU and 8GB of RAM. Experiments in section \ref{subsection_Run time comparison in distributed memory machines} were performed on multiple nodes of the NERSC machine Edison, where each node has two 12-core Intel processors. Codes in sections \ref{subsection_Approximation quality comparison on datasets}, \ref{subsection_Comparison on a pathological matrix: the Kahan matrix} were in Fortran and based on LAPACK. Codes in section \ref{subsection_Run time comparison in distributed memory machines} were in Fortran and C and based on ScaLAPACK. Codes in section \ref{subsection_SRQR based CUR and CX matrix decomposition algorithms} were in Matlab. 

Codes are available at \textrm{https://math.berkeley.edu/$\sim$jwxiao/}. 

\subsection{Approximation quality comparison on datasets}\label{subsection_Approximation quality comparison on datasets}
In this section we compare the approximation quality of the low-rank approximations computed by SRQR, QRCP and QR on practical datasets. We only list the results on two of them: Human Activities and Postural Transitions (HAPT) \cite{reyes2016transition} and the MNIST database (MNIST) \cite{lecun1998gradient}, while similar performance can be observed on others. We choose block size $b = 64$, oversampling size $p=10$, tolerance $g=5.0$ and set $l = k$ in our SRQR implementation. We compare the residual errors in figures \ref{Figure:Approximation quality comparison on HAPT} and \ref{Figure:Approximation quality comparison on MNIST}, where QR is not doing a great job, while QRCP and SRQR are doing equally well. We compare the run time in figures \ref{Figure:Run time comparison on HAPT} and \ref{Figure:Run time comparison on MNIST}, where SRQR is much faster than QRCP and close to QR. 

In other words, SRQR computes low-rank approximations comparable to those computed by QRCP in quality, yet at performance near that of QR.

\begin{figure}
\begin{minipage}[!t]{0.5\linewidth}
\centering
\includegraphics[width=1.0\linewidth]{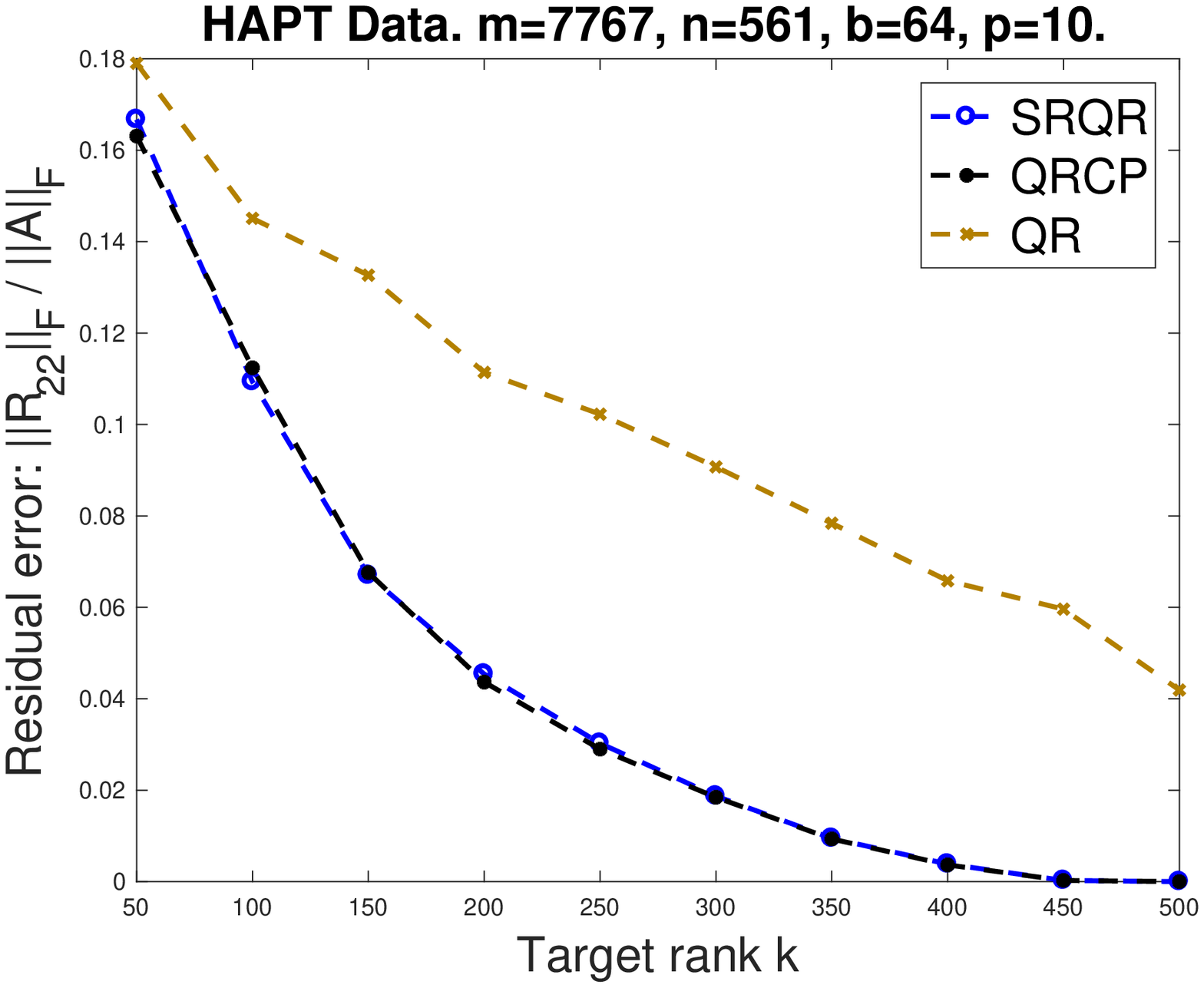}

\caption{Quality on HAPT}\label{Figure:Approximation quality comparison on HAPT}
\end{minipage}%
\begin{minipage}[!t]{0.5\linewidth}
\centering
\includegraphics[width=1.0\linewidth]{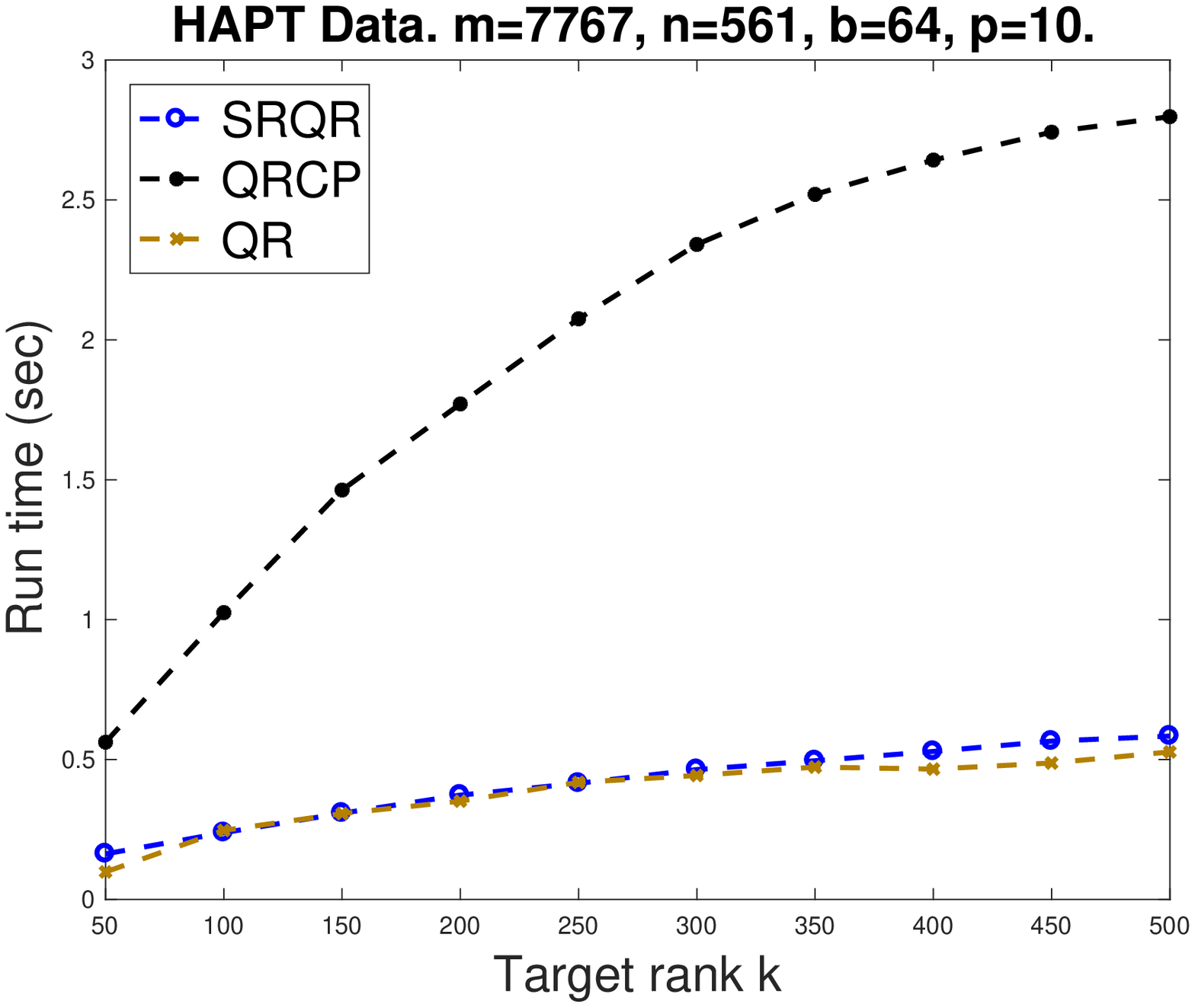}

\caption{Run time on HAPT}\label{Figure:Run time comparison on HAPT}
\end{minipage}
\begin{minipage}[!t]{0.5\linewidth}
\centering
\includegraphics[width=1.0\linewidth]{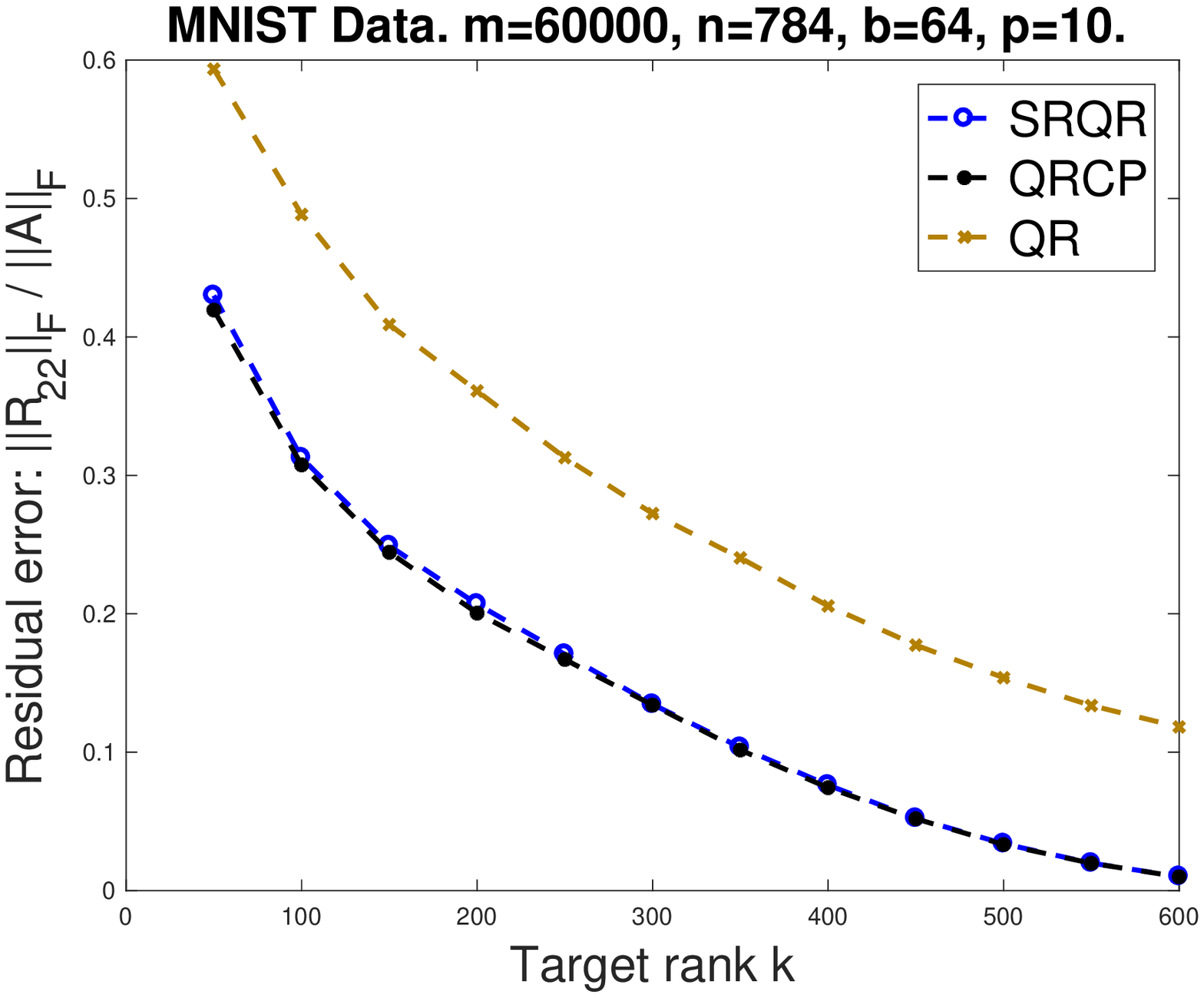}

\caption{Quality on MNIST}\label{Figure:Approximation quality comparison on MNIST}
\end{minipage}%
\begin{minipage}[!t]{0.5\linewidth}
\centering
\includegraphics[width=1.0\linewidth]{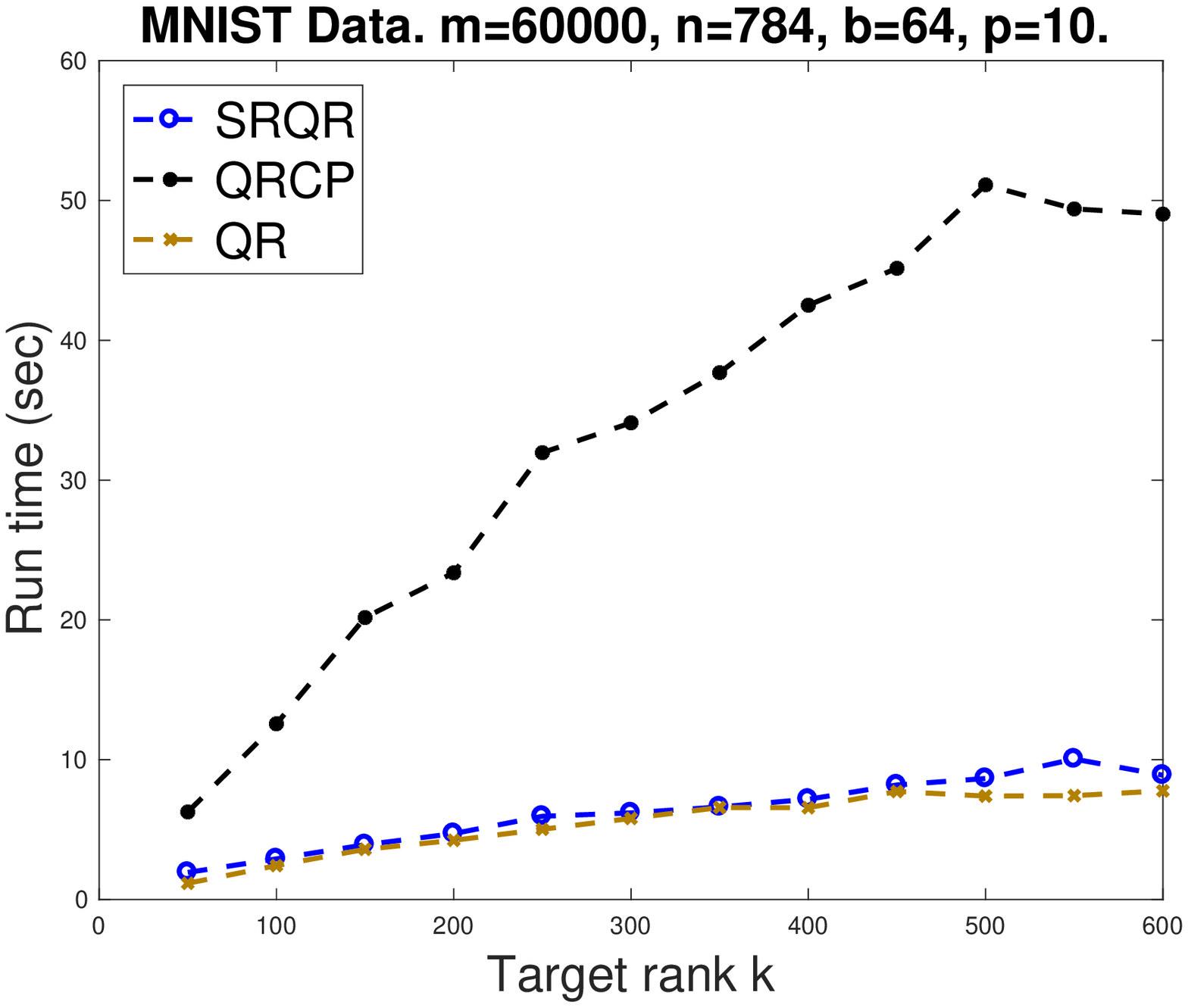}

\caption{Run time on MNIST}\label{Figure:Run time comparison on MNIST}
\end{minipage}
\end{figure}

\subsection{Comparison on a pathological matrix: the Kahan matrix}\label{subsection_Comparison on a pathological matrix: the Kahan matrix}
\begin{table}[!t]
\centering
\begin{tabular}{|c|c|c|c|c|}
\hline
n & k & SRQR $\left(\frac{\|R_{22}\|_F}{\|A\|_F}\right)$& QRCP $\left(\frac{\|R_{22}\|_F}{\|A\|_F}\right)$\\\hline
96 & 95 & 2.449E-13 & 1.808E-03\\\hline
192 & 191 & 1.031E-25 & 2.169E-05\\\hline
384 & 383 & 2.585E-50 & 4.414E-09\\\hline
\end{tabular}
\caption{Residual $\frac{\|R_{22}\|_F}{\|A\|_F}$ comparison on the Kahan matrix}
\label{table_Residual comparison on the Kahan matrix}

\centering
\begin{tabular}{|c|c|c|c|c|}
\hline
index & SRQR $\left(\frac{\sigma_{j}(R_{11})}{\sigma_j(A)}\right)$ & QRCP $\left(\frac{\sigma_{j}(R_{11})}{\sigma_j(A)}\right)$\\\hline
187 & 1.000 & 0.9942\\\hline
188 & 1.000 & 0.9932\\\hline
189 & 1.000 & 0.9916\\\hline
190 & 1.000 & 0.9883\\\hline
191 & 1.000 & 0.2806E-17\\\hline
\end{tabular}
\caption{Singular value approximation ratio $\frac{\sigma_{j}(R_{11})}{\sigma_j(A)}$}
\label{table_Singular value approximation ratio}
\end{table}
In this section we compare SRQR and QRCP on the Kahan matrix \cite{kahan1966numerical}. For the Kahan matrix, QRCP won't do any columns interchanges so it's equivalent to QR. We choose $c = 0.285, s = \sqrt{0.9999-c^2}, n = 96, 192, 384$ and $k = n - 1$. We choose block size $b = 64$, oversampling size $p=10$, tolerance $g=5.0$ and set $l = k$ in our SRQR implementation. From the relative residual errors summarized in table \ref{table_Residual comparison on the Kahan matrix}, we can see that SRQR is able to compute a much better low-rank approximation. 

The singular value ratios $\frac{\sigma_{j}(R_{11})}{\sigma_j(A)}$ never exceeds $1$ for any approximation, but we would like them to be close to $1$ for a reliable spectrum-revealing QR factorization. For the Kahan matrix where $n = 192$ and $k = 191$, table \ref{table_Singular value approximation ratio} demonstrates that QRCP failed to do so for the index 191 singular value, whereas SRQR succeeded for all singular values. 

The additional run time required to compute $g_2$ is negligible. In our extensive computations with practical data in machine learning and other applications, $g_2$ always remains modest and never triggers subsequent SRQR column swaps. Nonetheless, computing $g_2$ serves as an insurance policy against potential SRQR mistakes by QRCP or RQRCP. 

\subsection{Run time comparison in distributed memory machines}\label{subsection_Run time comparison in distributed memory machines}
In this section we compare run time and strong scaling of RQRCP against ScaLAPACK QRCP routines (PDGEQPF and PDGEQP3), ScaLAPACK QR routine PDGEQRF on distributed memory machines. PDGEQP3 \cite{quintana1998blas} is not yet incorporated into ScaLAPACK, but it's usually more efficient than PDGEQPF, so is also included in the comparison. 

The way the data is distributed over the memory hierarchy of a computer is of fundamental importance to load balancing and software reuse. ScaLAPACK uses a block cyclic data distribution in order to reduce overhead due to load imbalance and data movement. Block-partitioned algorithms are used to maximize local processor performance and ensure high levels of data reuse. 

Now we discuss how we parallelize RQRCP on a distributed memory machine based on ScaLAPACK. After we distribute $A$ to all processors, we use PDGEMM to compute $B = \Omega A$. In each loop, we use our version of PDGEQPF to compute a partial QRCP factorization of $B$, meanwhile we swap the columns of $A$ according to the pivots found on $B$. In our implementation, $A$ and $B$ share the same column blocking factor NB, therefore we don't introduce much extra communication costs since the same column processors are sending and receiving messages while doing the swaps on both $A$ and $B$. After we swap the pivoted columns to the leading position of the trailing matrix of $A$, we use PDGEQRF to perform a panel QR. Next, we use PDLARFT and PDLARFB to apply the transpose of an orthogonal matrix in a block form to the trailing matrix of $A$. At the end of each loop, we update the remaining columns of $B$ using updating formula \eqref{Eqn:updating_1}. See algorithm \ref{algorithm_Parallel RQRCP}.

\begin{algorithm}
\caption{Parallel RQRCP}\label{algorithm_Parallel RQRCP}
\begin{algorithmic}
\STATE $\textbf{Inputs:}$
\STATE $A$ is $m \times n$ matrix, $k$ is target rank, $1 \le k \le \min\left(m,n\right)$
\STATE $\textbf{Outputs:}$
\STATE $Q$ is $m \times m$ orthogonal matrix
\STATE $R$ is $m \times n$ upper trapezoidal matrix
\STATE $\Pi$ is $n \times n$ permutation matrix such that $A\Pi = QR$
\STATE $\textbf{Algorithm:}$
\STATE Determine block size $b$ and oversampling size $p \ge 0$
\STATE Distribute $A$ to processors using block-cyclic layout
\STATE Generate i.i.d Gaussian matrix $\Omega \in \mathcal{N}(0,1)^{(b+p) \times m}$
\STATE Compute $B=\Omega A$ using PDGEMM, initialize $\Pi=I_{n}$
\FOR{$i=1:b:k$}
\STATE $b=\min\left(b,k-i+1\right)$
\STATE Run partial version of PDGEQPF (QRCP) on $\!B(:,i:n)$, meanwhile apply the swaps to $A(:,i:n)$ and $\Pi$
\STATE PDGEQRF (QR) on $A\left(i:m,i:i+b-1\right) = \widetilde{Q} \widetilde{R}$
\STATE Use PDLARFT and PDLARFB to apply $\widetilde{Q}^T$ in a blocked form to $A(i:m,i+b:n)$
\STATE Update $B(1:b,i+b:n)=B(1:b,i+b:n)-B(1:b,i:j+b-1)(A(i:i+b-1.i:i+b-1))^{-1}A(i:i+b-1,i+b:n)$
\ENDFOR
\STATE Q is the product of $\widetilde{Q}$, R = upper trapezoidal part of $A$
\end{algorithmic}
\end{algorithm}

We choose block size $b=64$ and oversampling size $p=10$ in our RQRCP implementation. For all routines, we use an efficient data distribution by setting distribution block size MB = NB = 64 and using a square processor grid, i.e., $P_r = P_c$, as recommended in \cite{blackford1997scalapack}. Since the run time is only dependent on the matrix size but not the actual magnitude of the entries, we do the comparison on random matrices with different sizes, with $n = 20000, 50000$ and $200000$. See figures \ref{Figure:run_time_20000} through \ref{Figure:strong_scaling_200000}. 

The run time of RQRCP is always much better than that of ScaLAPACK QRCP routines and relatively close to that of ScaLAPACK QR routine. For very large scale low-rank approximations with limited number of processors, distributed RQRCP is likely the method of choice. 

However, RQRCP remains less than ideally strong scaled. There are two possible ways to improve our parallel RQRCP algorithm and implementation in our future work. 
\begin{itemize}
\item One bottleneck of our RQRCP parallel implementation is communication cost incurred by partial QRCP on the compressed matrix $B$. These communication costs are negligible on share memory machines or in distributed memory machines with a relatively small number of nodes. On distributed memory machines with a large number of nodes, many of them will be idle during partial QRCP computations on $B$, causing the gap between RQRCP and PDGEQRF (QR) run time lines in figures \ref{Figure:run_time_50000} and \ref{Figure:run_time_200000}. The communication costs on $B$ can be possibly reduced by using QR with tournament pivoting \cite{demmel2015communication} in place of partial QRCP. 
\item Another possible improvement of our RQRCP parallel implementation is to replace PDGEQRF (QR) with Tall Skinny QR (TSQR) \cite{demmel2012communication} in the panel QR factorization. 
\end{itemize}

The parallelization of SRQR in ScaLAPACK is also in our future work. 

\begin{figure}
\begin{minipage}[!t]{0.5\linewidth}
\centering
\includegraphics[width=1.0\linewidth]{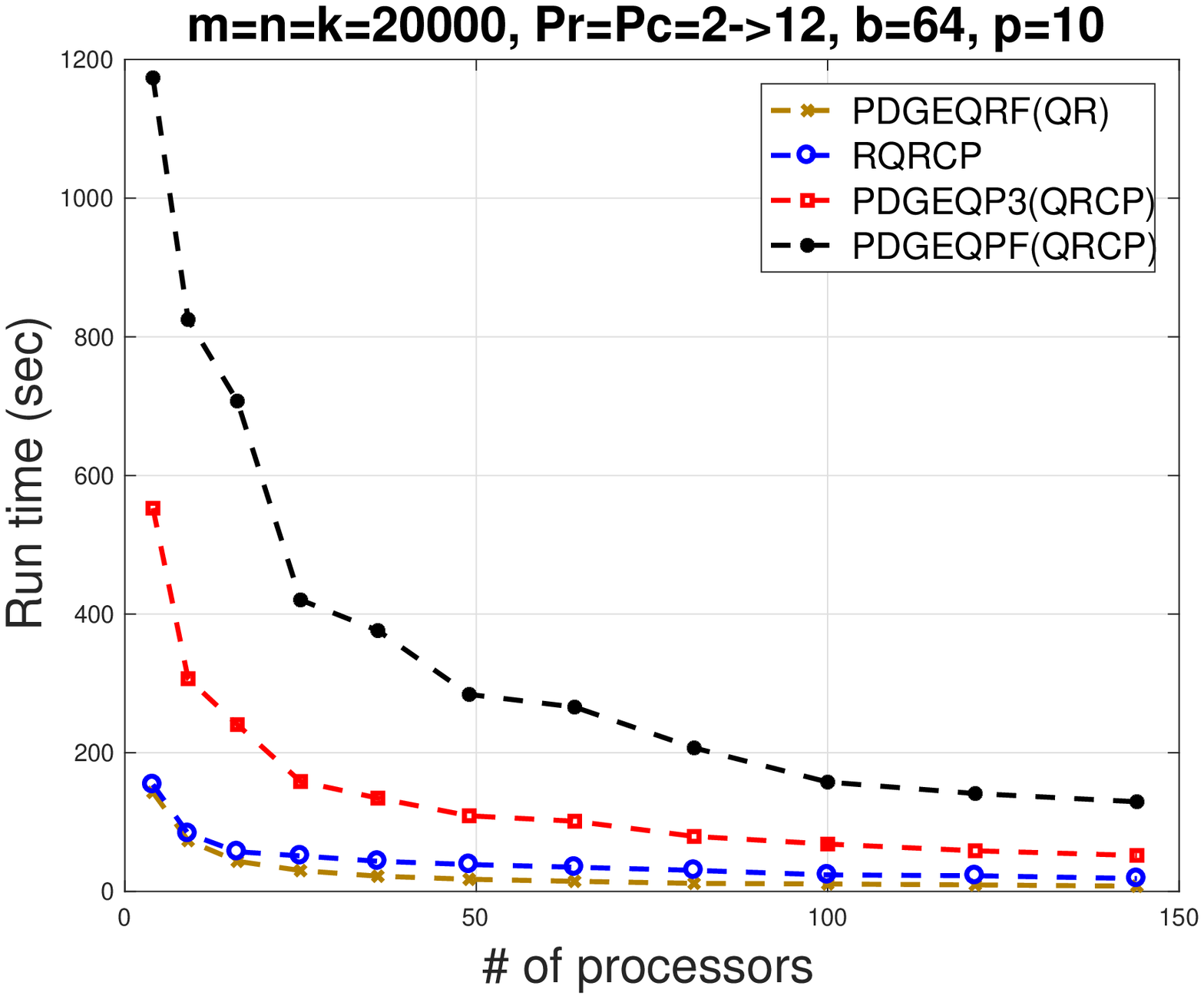}
\caption{Run time, n=20000}\label{Figure:run_time_20000}
\end{minipage}%
\begin{minipage}[!t]{0.5\linewidth}
\centering
\includegraphics[width=1.0\linewidth]{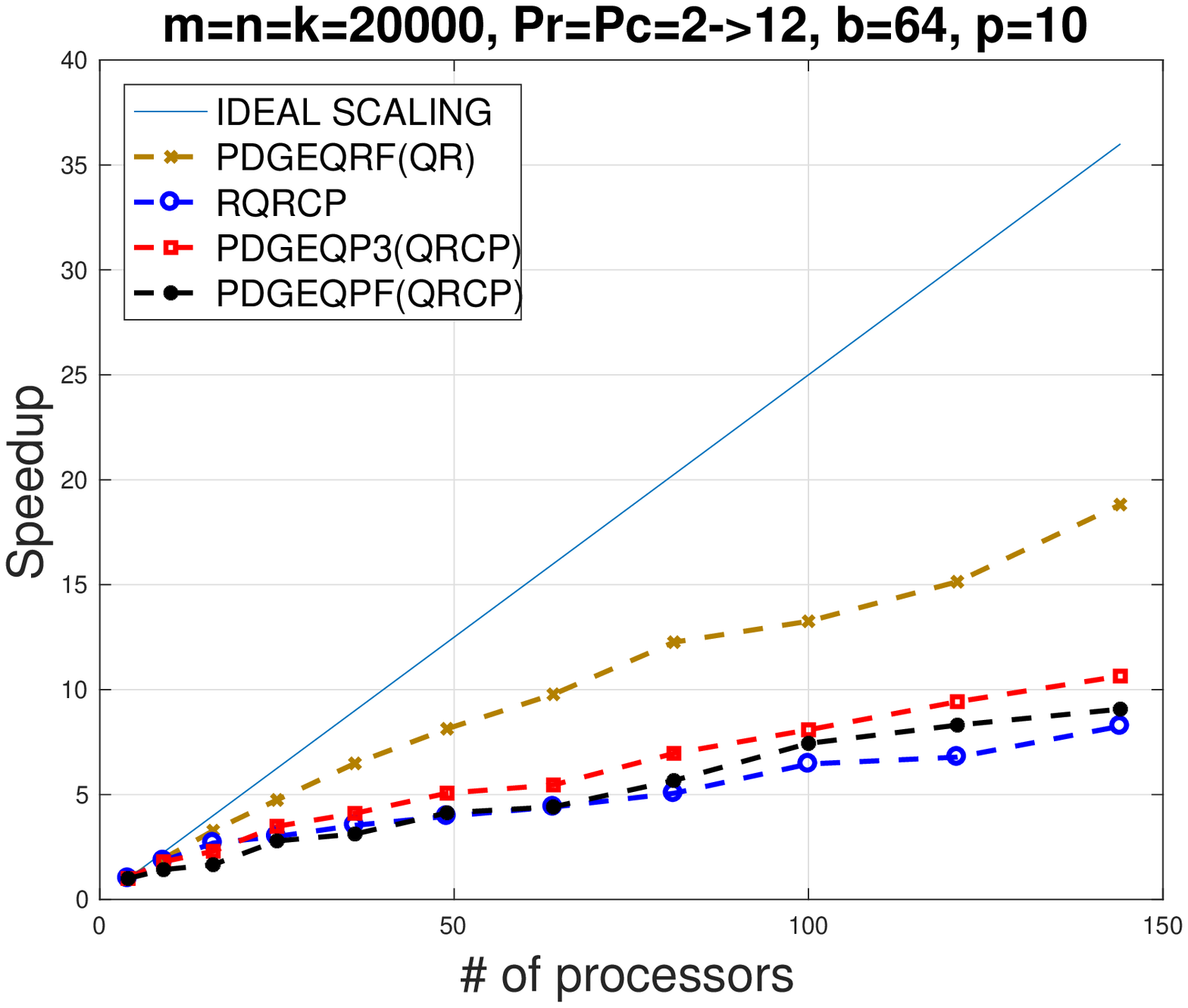}
\caption{Strong scaling, n=20000}\label{Figure:strong_scaling_20000}
\end{minipage}
\begin{minipage}[!t]{0.5\linewidth}
\centering
\includegraphics[width=1.0\linewidth]{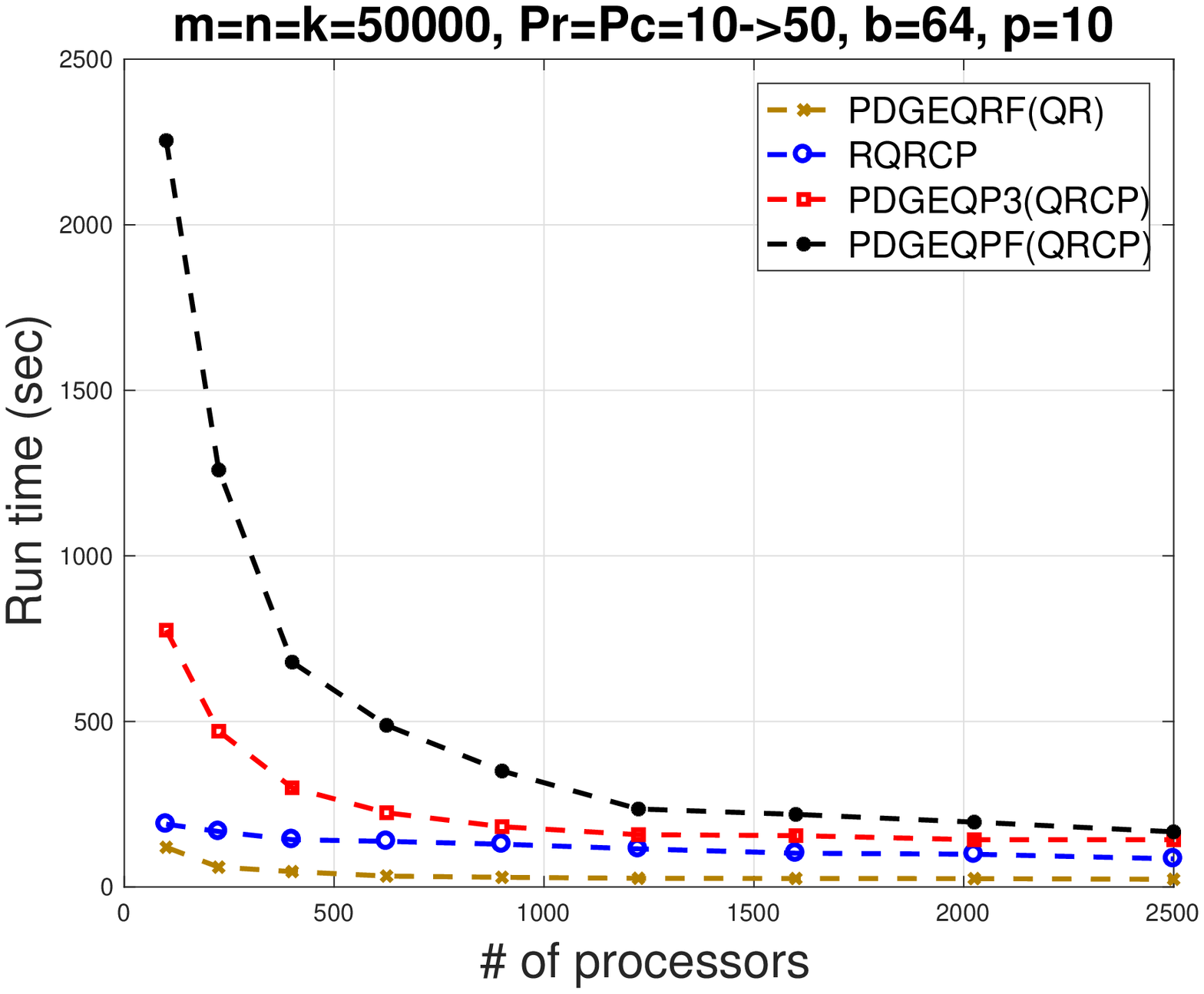}
\caption{Run time, n=50000}\label{Figure:run_time_50000}
\end{minipage}%
\begin{minipage}[!t]{0.5\linewidth}
\centering
\includegraphics[width=1.0\linewidth]{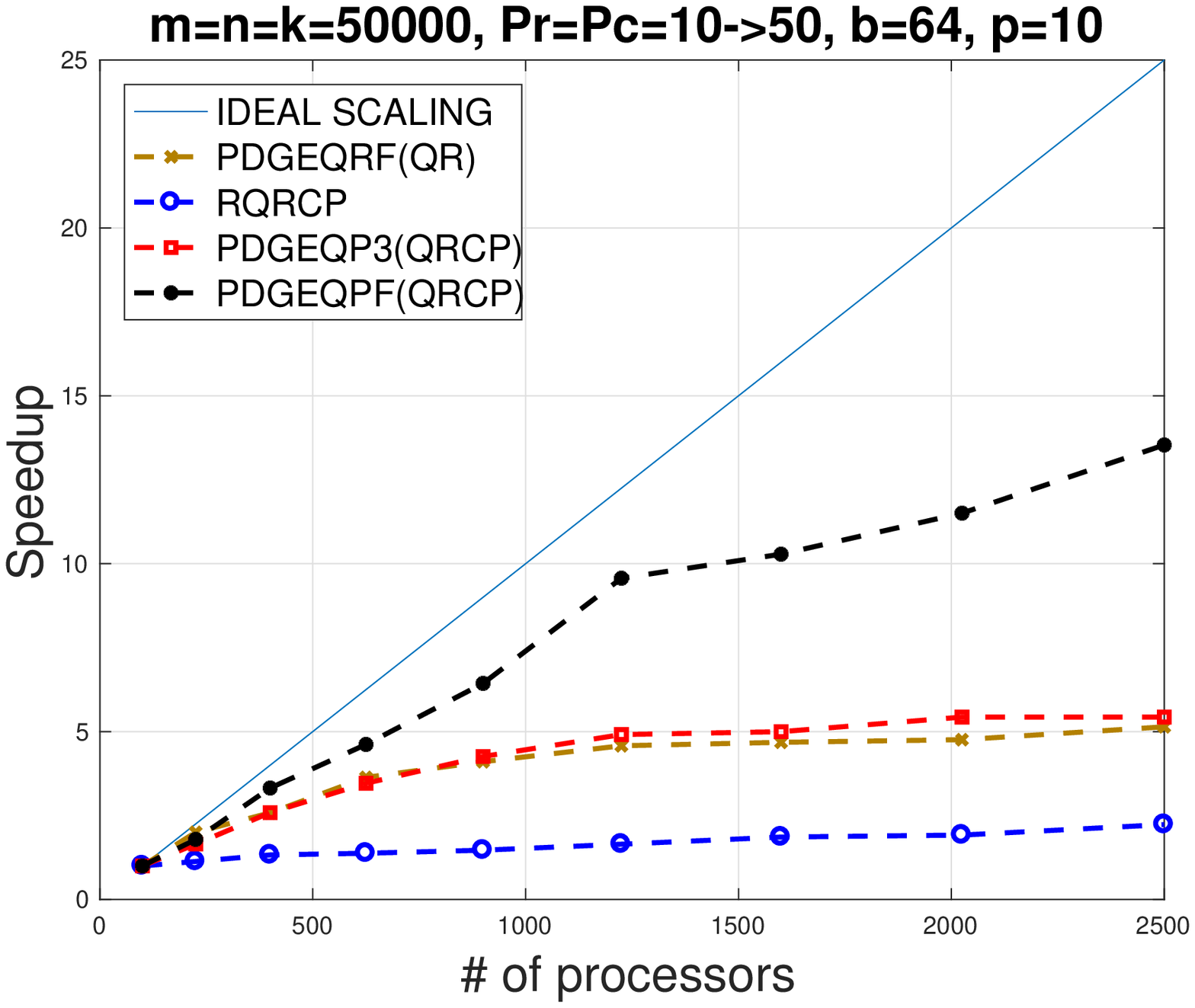}
\caption{Strong scaling, n=50000}\label{Figure:strong_scaling_50000}
\end{minipage}
\begin{minipage}[!t]{0.5\linewidth}
\centering
\includegraphics[width=1.0\linewidth]{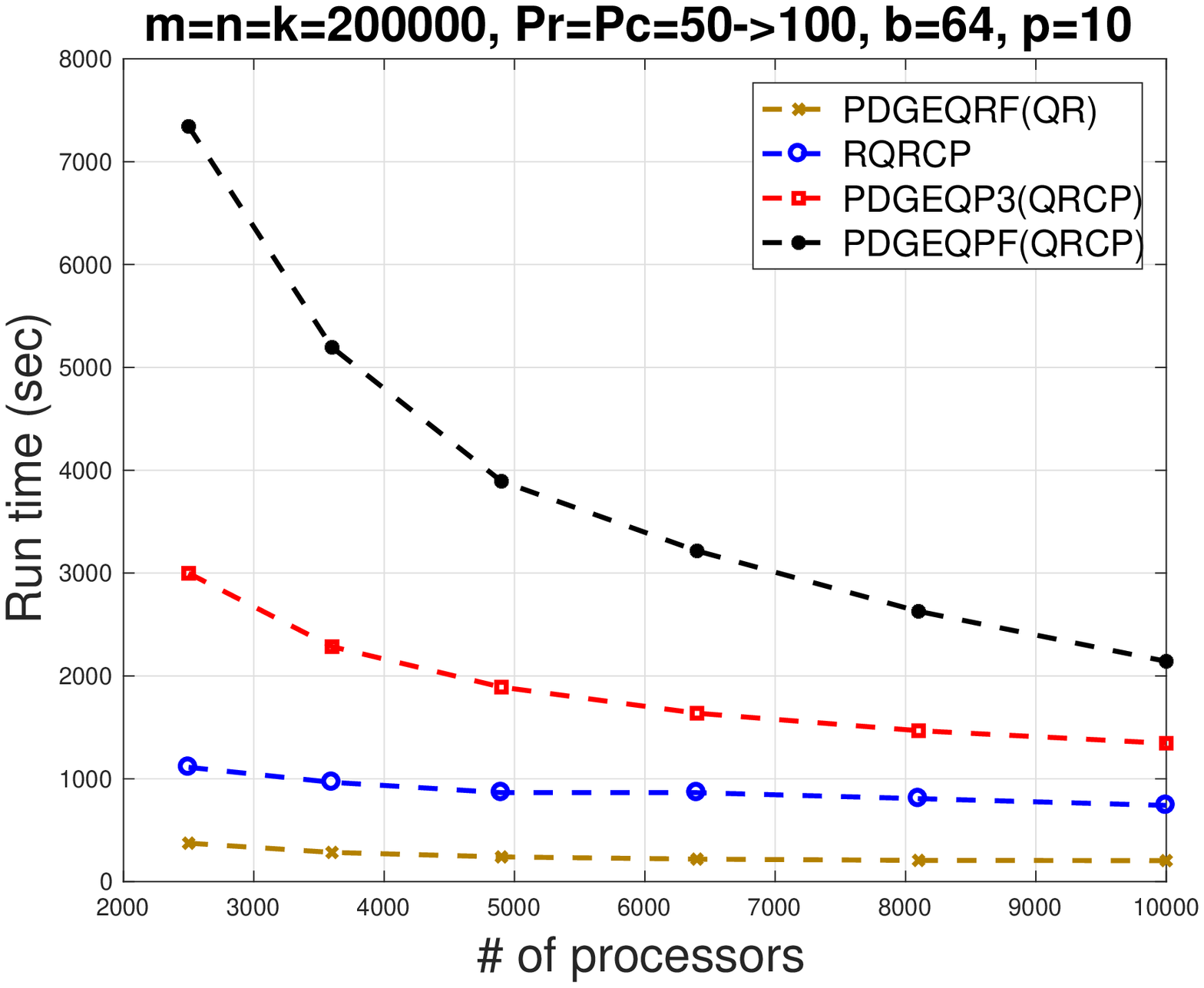}
\caption{Run time, n=200000}\label{Figure:run_time_200000}
\end{minipage}%
\begin{minipage}[!t]{0.5\linewidth}
\centering
\includegraphics[width=1.0\linewidth]{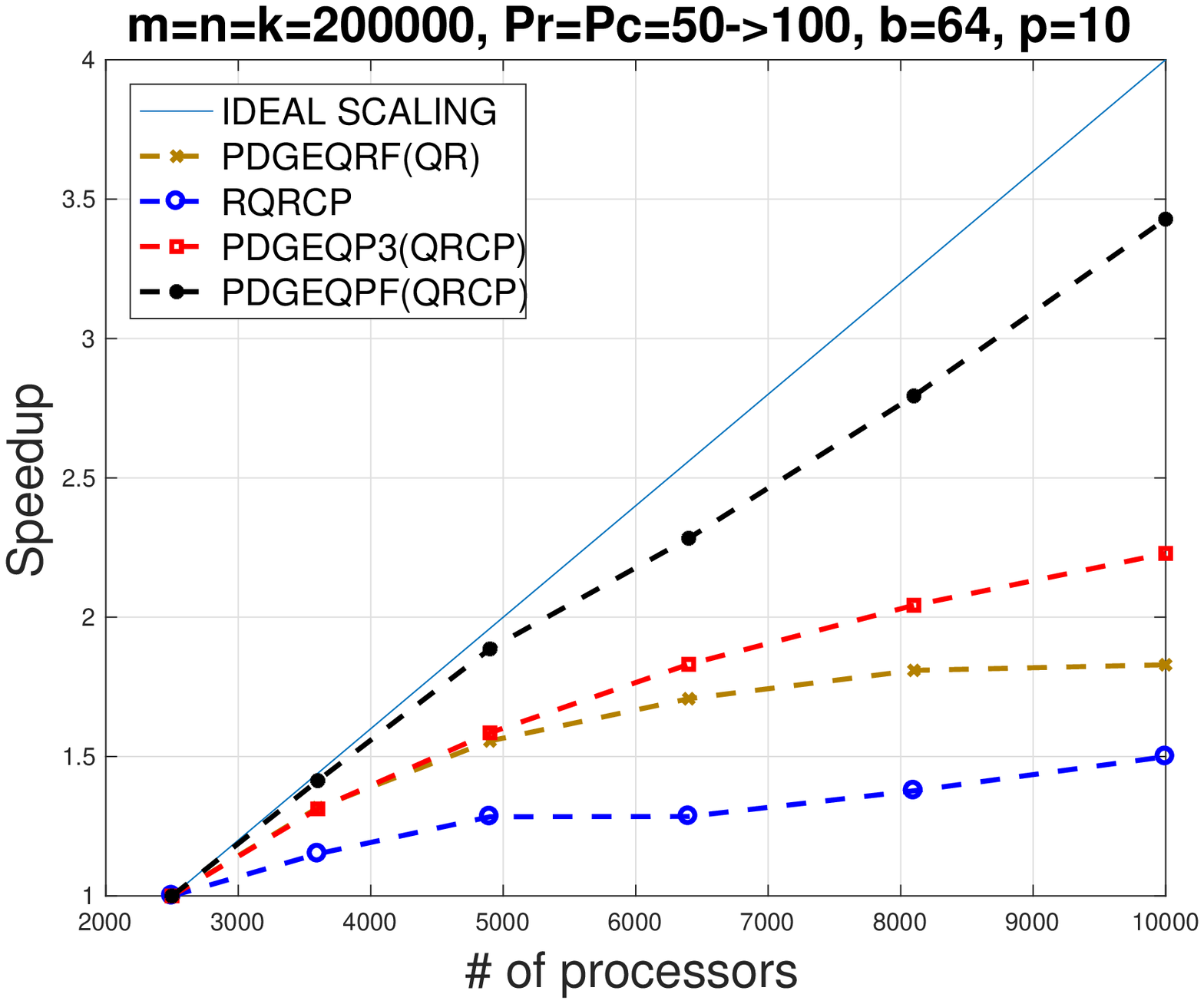}
\caption{Strong scaling, n=200000}\label{Figure:strong_scaling_200000}
\end{minipage}
\end{figure}

\subsection{SRQR based CUR and CX matrix decomposition}\label{subsection_SRQR based CUR and CX matrix decomposition algorithms}
The CUR and CX matrix decompositions are two important low-rank matrix approximation and data analysis techniques, and have been widely discussed in \cite{wang2013improving,gittens2013revisiting,boutsidis2014near}. A CUR matrix decomposition algorithm seeks to find $c$ columns of $A$ to form $C \in \mathbb{R}^{m \times c}$, $r$ rows of $A$ to form $R \in \mathbb{R}^{r \times n}$, and an intersection matrix $U \in \mathbb{R}^{c \times r}$ such that $\left\|A-CUR\right\|_F$ is small. One particular choice of $U$ is $C^{\dagger}AR^{\dagger}$, which is the solution to $\min_{X}\left\|CXR-A\right\|_F^2$. A CX decomposition algorithm seeks to find $c$ columns of $A$ to form $C \in \mathbb{R}^{m \times c}$ and a matrix $X \in \mathbb{R}^{c \times n}$ such that $\left\|A-CX\right\|_F$ is small. One particular choice of $X$ is $C^{\dagger}A$, which is the solution to $\min_{X}\left\|CX-A\right\|_F^2$. 

GitHub repository \cite{SimonDu2014} provides a Matlab library for CUR matrix decomposition. These CUR matrix decomposition algorithms can be modified to compute a CX matrix decomposition. Since the crucial component of CUR and CX matrix decompositions is column/row selection, we can use SRQR to find the pivots and hence compute these decompositions. In this experiment, we compare SRQR against the state-of-the-art CUR and CX matrix decomposition algorithms. 

We compare the approximation quality and run time on a kernel matrix $A$ of size $4177 \times 4177$ computed on Abalone Data Set \cite{bache2013uci}, for target rank $k = 200 = l$ with different numbers of columns and rows used. In Figures \ref{Figure:approximation quality cur}, \ref{Figure:run time cur}, \ref{Figure:approximation quality cx} and \ref{Figure:run time cx}, the x-axis stands for the number of columns and rows we choose for the CUR or CX matrix decomposition. The most efficient and effective method in the Matlab library is the near optimal method \cite{wang2013improving}. The near-optimal algorithm consists of three steps: the approximate SVD via random projection \cite{boutsidis2014near,halko2011finding}, the dual set sparsification algorithm \cite{boutsidis2014near}, and the adaptive sampling algorithm \cite{deshpande2006matrix}. We can see that SRQR and the near optimal method are obtaining much better low-rank approximations than all the other methods, while SRQR is much faster than the near optimal method. 

\begin{figure}
\begin{minipage}[!t]{0.5\linewidth}
\centering
\includegraphics[width=1.0\linewidth]{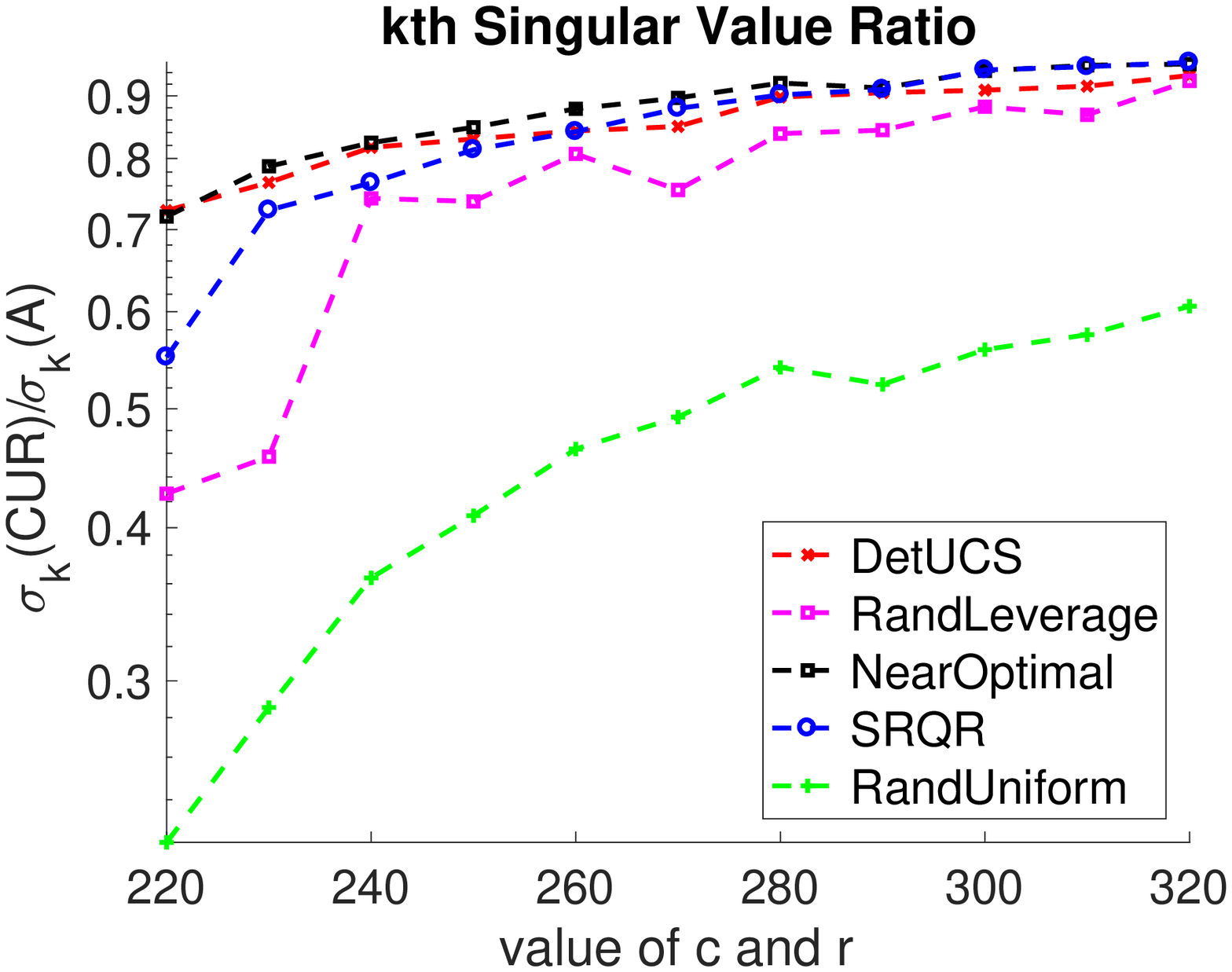}
\caption{Quality, CUR}\label{Figure:approximation quality cur}
\end{minipage}%
\begin{minipage}[!t]{0.5\linewidth}
\centering
\includegraphics[width=1.0\linewidth]{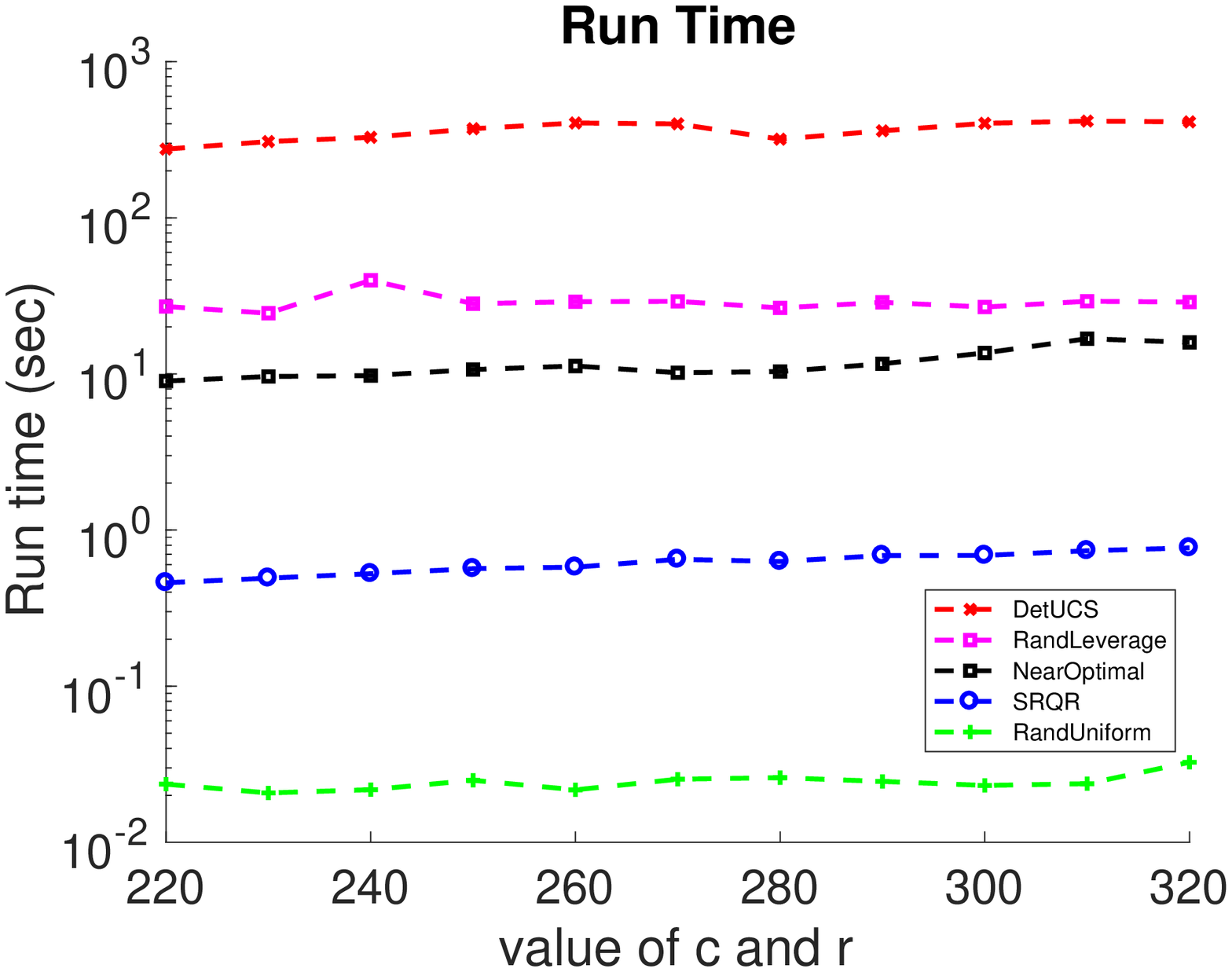}
\caption{Run time, CUR}\label{Figure:run time cur}
\end{minipage}
\begin{minipage}[!t]{0.5\linewidth}
\centering
\includegraphics[width=1.0\linewidth]{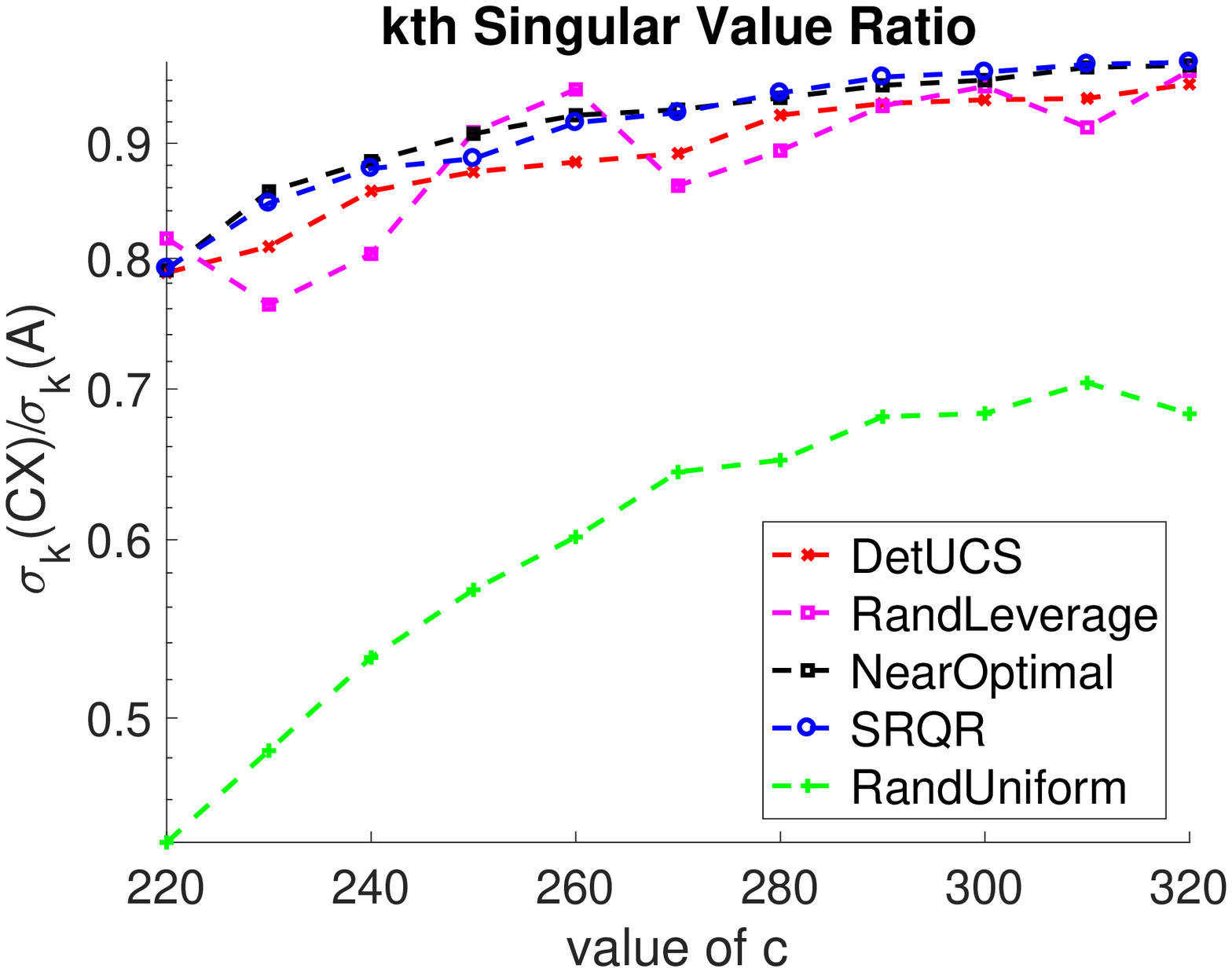}
\caption{Quality, CX}\label{Figure:approximation quality cx}
\end{minipage}%
\begin{minipage}[!t]{0.5\linewidth}
\centering
\includegraphics[width=1.0\linewidth]{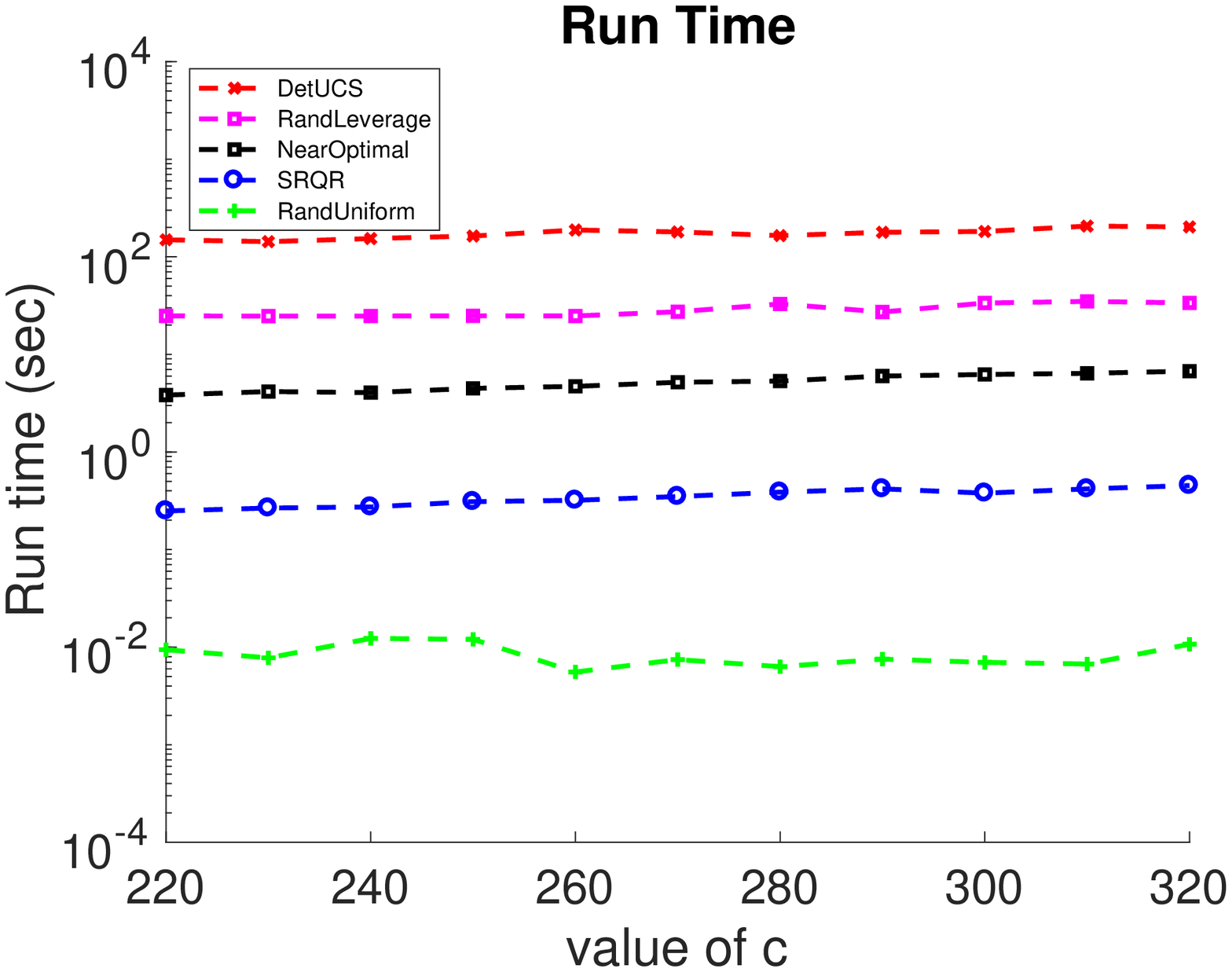}
\caption{Run time, CX}\label{Figure:run time cx}
\end{minipage}
\end{figure}

\section{Conclusion}
In this work, we showed that RQRCP is as reliable as QRCP with failure probabilities exponentially decaying in oversampling size. We developed spectrum-revealing QR factorizations (SRQR) for low-rank matrix approximations, and analyzed RQRCP as a reliable tool for such approximations. Most importantly, we report results from our distributed memory RQRCP implementations that are significantly better than QRCP implementations in ScaLAPACK, potentially making RQRCP a method of choice for large scale low-rank matrix approximations on distributed memory systems. We also developed SRQR based CUR and CX matrix decomposition algorithms, which are comparable to other state-of-the-art CUR and CX matrix decomposition algorithms in quality, while much more efficient in run time. 

\section*{Acknowledgment}
The authors would like to thank Gregorio Quintana-Orti for sharing the PDGEQP3 routine. The authors also would like to thank the reviewers for their helpful remarks.

\bibliographystyle{IEEEtran}
\bibliography{main.bib} 

\begin{thebibliography}{10}
\providecommand{\url}[1]{#1}
\csname url@samestyle\endcsname
\providecommand{\newblock}{\relax}
\providecommand{\bibinfo}[2]{#2}
\providecommand{\BIBentrySTDinterwordspacing}{\spaceskip=0pt\relax}
\providecommand{\BIBentryALTinterwordstretchfactor}{4}
\providecommand{\BIBentryALTinterwordspacing}{\spaceskip=\fontdimen2\font plus
\BIBentryALTinterwordstretchfactor\fontdimen3\font minus
  \fontdimen4\font\relax}
\providecommand{\BIBforeignlanguage}[2]{{%
\expandafter\ifx\csname l@#1\endcsname\relax
\typeout{** WARNING: IEEEtran.bst: No hyphenation pattern has been}%
\typeout{** loaded for the language `#1'. Using the pattern for}%
\typeout{** the default language instead.}%
\else
\language=\csname l@#1\endcsname
\fi
#2}}
\providecommand{\BIBdecl}{\relax}
\BIBdecl

\bibitem{anderson1999lapack}
E.~Anderson, Z.~Bai, C.~Bischof, S.~Blackford, J.~Dongarra, J.~Du~Croz,
  A.~Greenbaum, S.~Hammarling, A.~McKenney, and D.~Sorensen, \emph{LAPACK
  Users' guide}.\hskip 1em plus 0.5em minus 0.4em\relax Siam, 1999, vol.~9.

\bibitem{blackford1997scalapack}
L.~S. Blackford, J.~Choi, A.~Cleary, E.~D'Azevedo, J.~Demmel, I.~Dhillon,
  J.~Dongarra, S.~Hammarling, G.~Henry, A.~Petitet \emph{et~al.},
  \emph{ScaLAPACK users' guide}.\hskip 1em plus 0.5em minus 0.4em\relax siam,
  1997, vol.~4.

\bibitem{xiao2016spectrum}
J.~Xiao and M.~Gu, ``Spectrum-revealing cholesky factorization for kernel
  methods,'' in \emph{Data Mining (ICDM), 2016 IEEE 16th International
  Conference on}.\hskip 1em plus 0.5em minus 0.4em\relax IEEE, 2016, pp.
  1293--1298.

\bibitem{su2014color}
Q.~Su, Y.~Niu, G.~Wang, S.~Jia, and J.~Yue, ``Color image blind watermarking
  scheme based on qr decomposition,'' \emph{Signal Processing}, vol.~94, pp.
  219--235, 2014.

\bibitem{clarkson2013low}
K.~L. Clarkson and D.~P. Woodruff, ``Low rank approximation and regression in
  input sparsity time,'' in \emph{Proceedings of the forty-fifth annual ACM
  symposium on Theory of computing}.\hskip 1em plus 0.5em minus 0.4em\relax
  ACM, 2013, pp. 81--90.

\bibitem{demmel2015communication}
J.~W. Demmel, L.~Grigori, M.~Gu, and H.~Xiang, ``Communication avoiding rank
  revealing qr factorization with column pivoting,'' \emph{SIAM Journal on
  Matrix Analysis and Applications}, vol.~36, no.~1, pp. 55--89, 2015.

\bibitem{gu2015subspace}
M.~Gu, ``Subspace iteration randomization and singular value problems,''
  \emph{SIAM Journal on Scientific Computing}, vol.~37, no.~3, pp.
  A1139--A1173, 2015.

\bibitem{duersch2015true}
J.~A. Duersch and M.~Gu, ``True blas-3 performance qrcp using random
  sampling,'' \emph{arXiv preprint arXiv:1509.06820}, 2015.

\bibitem{gu1996efficient}
M.~Gu and S.~C. Eisenstat, ``Efficient algorithms for computing a strong
  rank-revealing qr factorization,'' \emph{SIAM Journal on Scientific
  Computing}, vol.~17, no.~4, pp. 848--869, 1996.

\bibitem{kahan1966numerical}
W.~Kahan, ``Numerical linear algebra,'' \emph{Canadian Math. Bull}, vol.~9,
  no.~6, pp. 757--801, 1966.

\bibitem{martinsson2017householder}
P.-G. Martinsson, G.~Quintana~OrtÍ, N.~Heavner, and R.~van~de Geijn,
  ``Householder qr factorization with randomization for column pivoting
  (hqrrp),'' \emph{SIAM Journal on Scientific Computing}, vol.~39, no.~2, pp.
  C96--C115, 2017.

\bibitem{dasgupta2003elementary}
S.~Dasgupta and A.~Gupta, ``An elementary proof of a theorem of johnson and
  lindenstrauss,'' \emph{Random structures and algorithms}, vol.~22, no.~1, pp.
  60--65, 2003.

\bibitem{reyes2016transition}
J.-L. Reyes-Ortiz, L.~Oneto, A.~Sama, X.~Parra, and D.~Anguita,
  ``Transition-aware human activity recognition using smartphones,''
  \emph{Neurocomputing}, vol. 171, pp. 754--767, 2016.

\bibitem{lecun1998gradient}
Y.~LeCun, L.~Bottou, Y.~Bengio, and P.~Haffner, ``Gradient-based learning
  applied to document recognition,'' \emph{Proceedings of the IEEE}, vol.~86,
  no.~11, pp. 2278--2324, 1998.

\bibitem{quintana1998blas}
G.~Quintana-Ort{\'\i}, X.~Sun, and C.~H. Bischof, ``A blas-3 version of the qr
  factorization with column pivoting,'' \emph{SIAM Journal on Scientific
  Computing}, vol.~19, no.~5, pp. 1486--1494, 1998.

\bibitem{demmel2012communication}
J.~Demmel, L.~Grigori, M.~Hoemmen, and J.~Langou, ``Communication-optimal
  parallel and sequential qr and lu factorizations,'' \emph{SIAM Journal on
  Scientific Computing}, vol.~34, no.~1, pp. A206--A239, 2012.

\bibitem{wang2013improving}
S.~Wang and Z.~Zhang, ``Improving cur matrix decomposition and the nystr{\"o}m
  approximation via adaptive sampling,'' \emph{The Journal of Machine Learning
  Research}, vol.~14, no.~1, pp. 2729--2769, 2013.

\bibitem{gittens2013revisiting}
A.~Gittens and M.~W. Mahoney, ``Revisiting the nystr{\"o}m method for improved
  large-scale machine learning,'' \emph{J. Mach. Learn. Res}, vol.~28, no.~3,
  pp. 567--575, 2013.

\bibitem{boutsidis2014near}
C.~Boutsidis, P.~Drineas, and M.~Magdon-Ismail, ``Near-optimal column-based
  matrix reconstruction,'' \emph{SIAM Journal on Computing}, vol.~43, no.~2,
  pp. 687--717, 2014.

\bibitem{SimonDu2014}
SimonDu, ``Cur-matrix-decomposition,''
  \url{https://github.com/SimonDu/CUR-matrix-decomposition}, 2014.

\bibitem{bache2013uci}
K.~Bache and M.~Lichman, ``Uci machine learning repository [http://archive.
  ics. uci. edu/ml]. irvine, ca: University of california, school of
  information and computer science. begleiter, h. neurodynamics laboratory.
  state university of new york health center at brooklyn. ingber, l.(1997).
  statistical mechanics of neocortical interactions: Canonical momenta
  indicatros of electroencephalography,'' \emph{Physical Review E}, vol.~55,
  pp. 4578--4593, 2013.

\bibitem{halko2011finding}
N.~Halko, P.-G. Martinsson, and J.~A. Tropp, ``Finding structure with
  randomness: Probabilistic algorithms for constructing approximate matrix
  decompositions,'' \emph{SIAM review}, vol.~53, no.~2, pp. 217--288, 2011.

\bibitem{deshpande2006matrix}
A.~Deshpande, L.~Rademacher, S.~Vempala, and G.~Wang, ``Matrix approximation
  and projective clustering via volume sampling,'' in \emph{Proceedings of the
  seventeenth annual ACM-SIAM symposium on Discrete algorithm}.\hskip 1em plus
  0.5em minus 0.4em\relax Society for Industrial and Applied Mathematics, 2006,
  pp. 1117--1126.

\end{thebibliography}
\end{document}